%% file: arXiv_R1.tex
\documentclass[USenglish, a4paper, oneside, 10pt]{article}

\input{Preamble.tex}

	\title{%
		Linesearch-free adaptive Bregman proximal gradient for convex minimization under local relative smoothness\thanks{%
			P. Latafat is a member of the Gruppo Nazionale per l'Analisi Matematica, la Probabilit\`a e le loro Applicazioni (GNAMPA - National Group for Mathematical Analysis, Probability and their Applications) of the Istituto Nazionale di Alta Matematica (INdAM - National Institute of Higher Mathematics).
			A. Themelis was supported by the JSPS KAKENHI grant number JP24K20737.
		}%
	}

	\author{%
		Hongjia Ou\thanks{%
			Faculty of Information Science and Electrical Engineering (ISEE), Kyushu University,
			744 Motooka, Nishi-ku, Fukuoka 819-0395, Japan.
			{\it E-mails:} {\sf ou.honjia.069@s.kyushu-u.ac.jp}, {\sf andreas.themelis@ees.kyushu-u.ac.jp}%
		}\and
		Puya Latafat\thanks{%
			IMT School for Advanced Studies Lucca.
			{\it E-mail:} {\sf puya.latafat@imtlucca.it}%
		}%
		\and
		Andreas Themelis\footnotemark[2]%
	}
	\date{}

\begin{document}

	\maketitle

	\begin{abstract}
		This paper introduces adaptive Bregman proximal gradient algorithms for solving convex composite minimization problems without relying on global relative smoothness or strong convexity assumptions.
		Building upon recent advances in adaptive stepsize selections, the proposed methods generate stepsizes based on local curvature estimates, entirely eliminating the need for backtracking linesearch.
		A key tool in our analysis is a Bregman generalization of Young's inequality, which allows the control of a critical inner product in terms of the same Bregman distances used in the updates.
		Our theory applies to problems where the differentiable term is merely \emph{locally} smooth relative to a distance-generating function, without requiring the existence of global moduli or symmetry coefficients.
		Numerical experiments demonstrate their competitive performance compared to existing approaches across various problem classes.
	\end{abstract}



	\tableofcontents

	%% ██ 1. Introduction ████████████████████████████████████████████████████████████████████████████████████████████████
	\section{Introduction}
		This work considers structured optimization problems of the form
		\[\label{eq:P} \tag{\text{P}}%
			\minimize_{x\in\overline{C}}
			\varphi(x)
			\coloneqq
			f(x) + g(x),
		\]
		where \(\overline{C}\) denotes the closure of \(C \coloneqq \interior\dom\kernel\) for a proper \(1\)-coercive function \(\kernel:\R^n\to\Rinf\) of Legendre type, \(f:\R^n\to\Rinf\), \(g:\R^n\to\Rinf\)  are proper closed convex, while \(f\) is only assumed to be \emph{locally} smooth relative to the \emph{distance-generating function} (dgf).
		By local smoothness relative to \(\kernel\) we refer to the existence, for any compact and convex set \(\K\subset\interior\dom\kernel\), of a constant \(L_{f,\K}^{\kernel}\geq0\) such that \(L_{f,\K}^{\kernel}\kernel-f\) is convex over \(\K\); see \cref{ass:basic}.
		When such a constant \(L_f^{\kernel}\) exists independently of \(\K\), that is such that \(L_f^{\kernel}\kernel-f\) is convex on \(\interior\dom\kernel\), then the usual notion of (global) smoothness of \(f\) relative to \(\kernel\) is recovered.

		For convex \(f\), these definitions are respective generalizations of local and global Lipschitz-smoothness of \(f\), that is, local and global Lipschitz continuity of \(\nabla f\).
		Indeed, these are recovered when \(\kernel=\j\), where with
		\[
			\j\coloneqq\tfrac12\norm{{}\cdot{}}^2
		\]
		we denote the squared Euclidean norm.

		A standard approach for solving \eqref{eq:P} is to use fixed point iterations with the Bregman proximal gradient (BPG) operator
		\begin{equation}\label{eq:BPG}
			x^{k+1}
		=
			\argmin_{w\in\R^n}\left\{
				f(x^k)+\innprod{\nabla f(x^k)}{w-x^k}+g(w)+\tfrac{1}{\gamma_{k+1}}\D_{\kernel}(w,x^k)
			\right\},
		\end{equation}
		where \(\gamma_{k+1}>0\) is the stepsize parameter and
		\begin{equation}\label{eq:D}
			\D_{\kernel}(x, y)
		\coloneqq
			\begin{cases}
				\kernel(x) - \kernel(y) - \innprod{\nabla\kernel(y)}{x - y} & \text{if }(x,y)\in\dom\kernel\times\interior\dom\kernel
			\\
				\infty & \text{otherwise}
			\end{cases}
		\end{equation}
		is the Bregman distance associated to \(\kernel\).

		Noticing that \(\D_{\j}(x,y)=\frac{1}{2}\norm{x-y}^2\) is the squared Euclidean norm, BPG updates \eqref{eq:BPG} reduce to standard proximal gradient iterations when \(\kernel=\j\).
		More generally, a well-chosen kernel \(\kernel\) can naturally encode the feasible region through its domain, blending the benefits of barrier and operator splitting methods.
		Beyond this, the Bregman framework addresses important smoothness limitations.
		In many applications, the differentiable term \(f\) lacks Lipschitz smoothness, ubiquitous requirement for first-order methods, but may exhibit smoothness \emph{relative} to a kernel \(\kernel\) other than the squared Euclidean norm \cite{bauschke2017descent,lu2018relatively}.
		Moreover, even in unconstrained problems and with \(f\) enjoying global Lipschitz smoothness, appropriate kernel selections can yield tighter smoothness parameters, enabling larger stepsizes and faster convergence.

		%% ░░░░ 1.1 Motivations and related work ░░░░░░░░░░░░░░░░░░░░░░░░░░░░░░░░░░░░░░░░░░░░░░░░░░░░░░░░░░░░░░░░░░░░░░░░░
		\subsection{Motivations and related work}
			Under the assumption that \(f\) is globally \(L_f^{\kernel}\)-smooth relative to \(\kernel\), the BPG method enjoys a descent property in terms of the Bregman distance, provided the stepsizes do not exceed \((1 + \alpha)/L_f^{\kernel}\), where \(\alpha\in[0,1]\) is the so-called \emph{symmetry coefficient} \cite{bauschke2017descent}.
			However, reliance on global smoothness and symmetry properties typically leads to unnecessarily small stepsizes and slow convergence in practice.
			This observation is not limited to the Bregman setting, as it also pertains the standard proximal gradient method and first-order algorithms in general.
			For this reason, even when constant stepsizes based on \emph{global} moduli are employable, time-varying selections reflecting the \emph{local} landscape of the problem can significantly improve algorithmic performance.

			\paragraph{Adaptive methods in the Euclidean setting}
				\emph{Backtracking linesearch} is a well-established practice to achieve this; linesearch refers to a trial-and-error procedure that iteratively adjusts the stepsize until a prescribed condition, typically a descent on the cost, is verified.
				These techniques can significantly accelerate convergence by selecting more effective stepsizes, but they incur higher per-iteration costs due to repeated evaluations, until the needed condition is met, thereby leading to slower individual iterations.
				In response to this, in a pioneering contribution, \cite{malitsky2020adaptive} introduced an adaptive stepsize selection strategy for the gradient method based on local estimates of Lipschitz moduli that can be derived from available data.
				This work sparked considerable interest in adaptive first-order methods and inspired a broad line of research, such as the proximal extension of \cite{latafat2024adaptive}, later further refined in several other flavors \cite{malitsky2024adaptive,latafat2024convergence,zhou2025adabb,ou2024safeguarding}.
				Among these developments, \cite{oikonomidis2024adaptive} showed that this class of adaptive methods extends to the local H\"olderian setting, thereby identifying an additional layer of adaptivity to unknown (local) H\"older exponents, while \cite{li2025simple,suh2025adaptive} proposed accelerated variants \`a la Nesterov \cite{nesterov1983method}.
				Interesting developments have also emerged in the stochastic setting \cite{aujol2025stochastic,ji2026stochastic}, as well as in the nonconvex regime \cite{yagishita2025simple}.
				Extensions beyond proximal gradient-type methods have likewise begun to emerge; see, e.g., \cite{jang2026alia} for Douglas-Rachford splitting and ADMM.

			\paragraph{Advances in the Bregman setting}
				All the above-mentioned works are however limited to the standard proximal gradient setup, dubbed ``Euclidean setting'' as it is captured by the choice of \(\kernel=\j\) as the square Euclidean norm.
				Linesearch techniques can be directly extended to Bregman proximal gradient iterations \eqref{eq:BPG}, where they preserve the same advantages and limitations as in the Euclidean case.
				In contrast, generalizing other adaptive schemes to the Bregman setting proves significantly more challenging.
				To the best of our knowledge, the only successful extension in this direction is the \BaGRA{} method proposed in \cite{tam2023bregman}, which adapts the golden-ratio scheme of \cite{malitsky2020golden} to the Bregman context.
				Remarkably, similarly to its predecessor \cite{tam2023bregman} covers a class of hemivariational inequalities broader than composite minimization.
				On the other hand, it requires the Bregman kernel \(\kernel\) to be strongly convex, with the stepsize parameters explicitly dependent on the corresponding modulus of strong convexity.
				More importantly, tailored to our setup this method addresses the \emph{unconstrained} minimization of \(f+g\) assuming that a solution exists within \(\dom\kernel\), and derives its stepsizes from \emph{Euclidean} (Lipschitz) estimates, which are not aligned with the underlying Bregman geometry.
				As we will demonstrate in our simulations, this results in conservative stepsize selections and, consequently, slower convergence in practice.

				Another notable exception is the very recent work \cite{malitsky2026entropic}, which focuses on the \emph{Boltzmann--Shannon entropy} as dgf \(\kernel\) and proposes an adaptive Bregman scheme for solving linear systems.
				In this setup, it generates very large stepsizes \`a la Polyak \cite{polyak1969minimization} which can significantly speedup convergence compared to a standard backtracking routine.
				With a minor modification, the scheme can also be applied to the minimization of Lipschitz-differentiable functions over the positive orthant, provided (a lower bound of) the optimal value is available.

		%% ░░░░ 1.2 Contribution ░░░░░░░░░░░░░░░░░░░░░░░░░░░░░░░░░░░░░░░░░░░░░░░░░░░░░░░░░░░░░░░░░░░░░░░░░░░░░░░░░░░░░░░░░
		\subsection{Contribution}
			In this work, we propose two adaptive stepsize selection strategies for Bregman proximal gradient iterations that operate without requiring strong convexity of the Bregman kernel or global Lipschitz smoothness of the differentiable term, thus significantly broadening the scope of applicability of Bregman-based methods.
			A central technical challenge in the Bregman setting arises from controlling inner product terms, which in the Euclidean case are typically bounded using Young's or Cauchy--Schwarz inequalities.
			A key tool in our approach is a Bregman generalization of Young's inequality, which enables a direct and effective handling of inner products in terms of Bregman distances without the need to resort to Euclidean-type bounds and assumptions.
			Several numerical experiments confirm that the proposed methodology is competitive with existing approaches in terms of convergence speed, robustness across problem classes, and employment of large stepsizes, all under very general working assumptions.
			As a byproduct, our proof technique also yields convergence guarantees for Bregman proximal gradient iterations with a standard linesearch, under the same general assumptions.

		%% ░░░░ 1.3 Preliminaries and notation ░░░░░░░░░░░░░░░░░░░░░░░░░░░░░░░░░░░░░░░░░░░░░░░░░░░░░░░░░░░░░░░░░░░░░░░░░░░
		\subsection{Preliminaries and notation}
			The set of natural numbers is \(\N\coloneqq\{0,1,2,\dots\}\), while \(\R\), \(\R_{++}\coloneqq(0,\infty)\), and \(\Rinf\coloneqq\R\cup\{\pm\infty\}\) denote the set of real, strictly positive, and extended-real numbers, respectively.
			For \(t\in\R\), we define \([t]_+\coloneqq\max\{t,0\}\).
			We use \(\innprod{{}\cdot{}}{{}\cdot{}}\) to denote the standard inner product on \(\R^n\), and for a symmetric and positive definite \(\R^{n\times n}\) matrix \(Q\), denoted \(Q\in\symm_{++}(\R^n)\), we let \(\norm{x}_Q=\sqrt{\innprod{x}{Qx}}\) be the induced norm.
			In case \(Q\) is the identity matrix, we simply write \(\norm{x}\).
			Given a set \(\mathcal D\subseteq\R^n\), with \(\interior \mathcal D\) we denote its interior.

			With \(\j:\R^n\to\R\) we indicate the square Euclidean norm
			\(
				\j(x)=\frac{1}{2}\norm{x}^2
			\).
			The \emph{domain} and \emph{epigraph} of an extended-real-valued function \(h:\R^n\to\Rinf\) are the sets
			\(
				\dom h
			\coloneqq
				\{x\in\R^n \mid h(x)<\infty\}
			\)
			and
			\(
				\epi h
			\coloneqq
				\{(x, c)\in\R^n\times\R \mid h(x)\leq c\}
			\).
			Function \(h\) is said to be: \emph{proper} if \(h>-\infty\) and \(\dom h\neq\emptyset\); \emph{lower semicontinuous (lsc)} if \(\epi h\) is a closed subset of \(\R^{n+1}\); \emph{1-coercive} if
			\(
				\lim_{\norm{x}\to\infty}\frac{h(x)}{\norm{x}}=\infty
			\).

			The \emph{conjugate} of a proper, lsc, convex function \(h:\R^n\to\Rinf\) is the proper, lsc, convex function \(h^*:\R^n\to\Rinf\) defined by
			\(
				h^*(\xi)
			\coloneqq
				\sup_{x\in\R^n}\{\innprod{\xi}{x}-h(x)\}
			\).
			The \emph{(convex) subdifferential} of \(h\) at \(x\in\dom h\) is the set
			\[
				\partial h(x)
			\coloneqq
				\{u\in\R^n \mid h(x')\geq h(x)+\innprod{u}{x'-x} ~ \forall x'\in\R^n\},
			\]
			while \(\partial h(x)=\emptyset\) for \(x\notin\dom h\).
			\(h\) is differentiable at \(x\) iff \(\partial h(x)\) is a singleton, and in this case one has that \(\partial f(x)=\{\nabla h(x)\}\).

			\subsubsection*{Bregman distance}
				We next list a few known facts related to Bregman distances as in \eqref{eq:D}.

				\begin{fact}[three-point identity {\cite[Lem. 3.1]{chen1993convergence}}]\label{thm:3p}%
					Let \(\kernel:\R^n\to\Rinf\) be a proper lsc convex function differentiable on \(\interior\dom\kernel\).
					For any \(x\in\dom\kernel\), and \(y,z\in\interior\dom\kernel\) the following holds:
					\[
						\D_{\kernel}(x,z)
					=
						\D_{\kernel}(x,y)
						+
						\D_{\kernel}(y,z)
						+
						\innprod{x-y}{\nabla\kernel(y)-\nabla\kernel(z)}.
					\]
				\end{fact}

				We say that a proper, lsc, convex function \(\kernel:\R^n\to\Rinf\) is \emph{of Legendre type} (or simply \emph{Legendre}) if it is
				(i) \emph{essentially smooth}, namely differentiable on \(\interior\dom\kernel\neq\emptyset\) and such that \(\norm{\nabla\kernel(x^k)}\to\infty\) whenever \(\interior\dom\kernel\ni x^k\to x\in\dom\kernel\setminus\interior\dom\kernel\), and
				(ii) \emph{essentially strictly convex}, namely strictly convex on every convex subset of \(\dom\partial\kernel\).

				\begin{fact}\label{fact:D}%
					Let \(\kernel:\R^n\to\Rinf\) be of Legendre type, and let \(C\coloneqq\interior\dom\kernel\).
					\begin{enumerate}
					\item \label{fact:D:coercive}%
							{\upshape \cite[Prop. 2.16]{bauschke1997legendre}, \cite[Thm. 3.7(vi)]{bauschke1997legendre}}
							If \(\kernel\) is 1-coercive, then \(\dom\kernel^*=\R^n\) and \(\D_{\kernel}(x,{}\cdot{})\) is coercive for any \(x\in C\).

					\item \label{fact:D:boundary}%
						{\upshape \cite[Thm. 3.8(i)]{bauschke1997legendre}}
						If a sequence \((x^k)_{k\in\N}\) converges to a boundary point of \(C\), then \(\D_{\kernel}(x,x^k)\to\infty\) for any \(x\in C\).

					\item \label{fact:D:solodov}%
						{\upshape \cite[Thm. 2.4]{solodov2000inexact}}
						Suppose that \(\dom\kernel\) is closed and that \(\D_{\kernel}(x^k,y^k)\to0\).
						If either \((x^k)_{k\in\N}\) or \((y^k)_{k\in\N}\) converges, then the other sequence too converges to the same limit.

					\item \label{fact:D:Dphi*}%
						{\upshape \cite[Thm. 26.5]{rockafellar1970convex}}
						The conjugate \(\kernel^*\) is continuously differentiable and strictly convex on \(\interior\dom\kernel^*\), with \(\nabla\kernel^* = \nabla\kernel^{-1}:\interior\dom\kernel^*\to\interior\dom\kernel\).

					\item \label{fact:D:DtoD*}%
						{\upshape \cite[Thm. 3.7(v)]{bauschke1997legendre}}
						\(\D_{\kernel}(x,y) = \D_{\kernel^*}(\nabla \kernel(y), \nabla \kernel(x))\) for any \(x,y\in C\).
					\end{enumerate}
				\end{fact}

				Finally, we introduce the symbol \(\DD_{\kernel}\) to indicate the \emph{symmetrized Bregman distance}
				\begin{equation}\label{eq:DD}
					\DD_{\kernel}(x,y)
				\coloneqq
					\D_{\kernel}(x,y)+\D_{\kernel}(y,x)
				=
					\begin{cases}
						\innprod{\nabla\kernel(x)-\nabla\kernel(y)}{x - y} & \hspace*{-7pt}\text{if } x,y\in\interior\dom\kernel
					\\
						\infty & \hspace*{-7pt}\text{otherwise.}
					\end{cases}
				\end{equation}

	%% ██ 2. Problem setting and proposed algorithms █████████████████████████████████████████████████████████████████████
	\section{Problem setting and proposed algorithms}
		Problem \eqref{eq:P} will be studied under the following assumptions.

		\begin{assumption}\label{ass:basic}%
			The following hold in \eqref{eq:P}:
			\begin{enumerate}
			\item \label{ass:h}%
				\(C=\interior\dom\kernel\) for a proper, convex, 1-coercive function \(\kernel\) of Legendre type, which is twice continuously differentiable with \(\nabla^2\kernel \succ 0\) on \(C\).
			\item \label{ass:f}%
				\(f:\R^n\to\Rinf\) is proper, convex, lsc, and \emph{locally smooth relative to \(\kernel\)}:
				that is, \(\dom\kernel\subseteq\dom f\), and for every convex and compact set \(\K\subset C\) there exists \(L_{f,\K}^{\kernel}>0\) such that \(L_{f,\K}^{\kernel}\kernel-f\) is convex on \(\K\).
			\item \label{ass:g}%
				\(g:\R^n\to\Rinf\) is proper, lsc, and convex with \(\dom g\cap C\neq\emptyset\).
			\item \label{ass:argmin}%
				A solution exists: \(\argmin_{\overline C}\varphi\neq\emptyset\).
			\end{enumerate}
		\end{assumption}

		Beyond convexity, all these basic requirements on \(f\) and \(g\) are virtually negligible.
		The local relative smoothness in \cref{ass:f} is tantamount to saying that \(f\) is differentiable on \(C\),\footnote{%
			See \cite[Prop. 3.7]{wang2024mirror} or \cite[Prop. 2.5]{ahookhosh2021bregman}.
		}
		and that
		\begin{equation}\label{eq:Df<D}
			\D_f(x, y)\leq L_{f,\K}^{\kernel}\D_{\kernel}(x, y)
		\quad
			\forall x,y\in\K,
		\end{equation}
		which simplifies to local Lipschitz continuity of \(\nabla f\) when \(\kernel=\tfrac{1}{2}\norm{x}^2\).

		\begin{remark}\label{rem:local}%
			The locality notion of \cref{ass:f} is weaker than the more natural local relative smoothness used, for instance, in the backtracking analysis of \cite[\S5.3]{li2018learning} which entails the existence of a modulus on every \emph{bounded} subset of \(\interior\dom\kernel\).
			Since such sets may accumulate at the boundary of \(\dom\kernel\), that condition implicitly enforces uniform control up to the boundary.
			In contrast, our assumption only requires control on \emph{compact} subsets and therefore allows the relative smoothness modulus to blow up near the boundary.
			\hfill\qedsymbol
		\end{remark}

		Indeed, when \(\nabla^2\kernel(x)\succ0\) as in \cref{ass:h}, any convex function \(f\) which is twice differentiable on \(C\) satisfies \cref{ass:f}, even if exhibiting an infinite slope at boundary points of \(C\).\footnote{%
			The assumption \(\nabla^2\kernel(x)\succ0\) prevents pathological situations far from the boundary: for instance, when \(\nabla^2\kernel(\bar x)=0\) while \(\nabla^2 f(\bar x)\neq0\), there exists no \(L\) such that \(L\kernel-f\) is convex in a neighborhood of \(\bar x\), having \(\nabla^2(L\kernel-f)(\bar x)=-\nabla^2 f(\bar x)\not\succeq0\) in this case.
		}

		\begin{example}\label{ex:locrelsmooth}%
			Let \(\kernel(x)=\sum_{i=1}^n(x_i\ln x_i-x_i)\) be the \emph{Boltzmann--Shannon entropy} on \(\R_+^n\).
			Any convex function that is twice differentiable on \(\R_{++}^{n}\) is locally smooth relative to \(\kernel\) as in \cref{ass:f}, despite the fact that there may exist no \(L\) such that \(L\kernel-f\) is convex around points on the boundary of \(\R_+^n\).
			\hfill\qedsymbol
		\end{example}

		As we detail in the following subsection, the basic requirements listed in \cref{ass:basic} are enough to guarantee that the proposed adaptive stepsize selection strategies produce iterates \(x^k\) such that \(\inf_{k\in\N}\varphi(x^k)=\inf_{\overline C}\varphi\).
		Slightly more can be said upon assuming that the Bregman distance generated by \(\kernel\) satisfies the following mild additional assumption.

		\begin{assumption}[Bregman with zone \(C\) {\cite[Def. 2.1]{solodov2000inexact}}]\label{ass:zone}%
			The dgf \(\kernel\) satisfies the following:
			\begin{enumerate}
			\item \label{ass:Dhto0}%
				\(\D_{\kernel}(x, x^k) \to 0\) whenever \(C \ni x^k \to x\) (in particular, \(\dom\kernel\) is closed).
			\item \label{ass:xkbounded}%
				\(\D_{\kernel}(x,{}\cdot{})\) is level bounded for any \(x\in\overline{C}\setminus C\).
			\end{enumerate}
		\end{assumption}

		\Cref{ass:zone} holds vacuously whenever \(\kernel\) has full domain \(\R^n\).
		More generally, it is a standard requirement satisfied by many separable kernels used in practice; see for instance \cite[\S6.1]{bauschke1997legendre}.
		The nonseparable case is more intricate, as \cite[Ex. 6.9]{bauschke1997legendre} showcases.
		We complement that example with a more practically relevant choice of \(\kernel\) for which \cref{ass:zone} fails.

		\begin{example}\label{ex:zone}%
			Consider the 1-coercive and Legendre function \(\kernel:\R^n\to\Rinf\) given by
			\[
				\kernel(x)
			=
				\begin{cases}
					-\sqrt{1-\norm{x}^2} & \text{if } \norm{x}\leq1
				\\
					\infty & \text{otherwise.}
				\end{cases}
			\]
			For \(n=1\), this yields the ``Hellinger divergence'' which complies with \cref{ass:zone}, and likewise so do the separable extensions \(x\mapsto\sum_{i=1}^N\phi(x_i)\) for any \(N\in\N\) \cite[Ex. 6.4]{bauschke1997legendre}.
			However, for \(n\geq2\) this is no longer the case.
			To see this, for \((x,y)\in\dom\kernel\times\interior\dom\kernel\) note that
			\[
				\D_\kernel(x,y)
			=
				\tfrac{1-\innprod{x}{y}}{\sqrt{1-\norm{y}^2}}
				-
				\sqrt{1-\norm{x}^2}.
			\]
			Let \(x=(1,0,\dots,0)\in\dom\kernel\) and \(y_t=\bigl(t,\sqrt{(1-t^2)(1-(1-t)^2)},0,\dots,0\bigr)\in\interior\dom\kernel\) for \(t\in[0,1)\); then, as \(t\to 1^-\) one has that \(y_t\to x\), yet not only isn't
			\[
				\D_\kernel(x,y_t)
			=
				\tfrac{1-t}{\sqrt{1-t^2-(1-t^2)(1-(1-t)^2)}}
			=
				\tfrac{1}{\sqrt{1-t^2}}
			\]
			vanishing, it actually diverges.
			\hfill\qedsymbol
		\end{example}

		%% ░░░░ 2.1 Local moduli estimates ░░░░░░░░░░░░░░░░░░░░░░░░░░░░░░░░░░░░░░░░░░░░░░░░░░░░░░░░░░░░░░░░░░░░░░░░░░░░░░░
		\subsection{Local moduli estimates}\label{sec:estimates}%
			Our approach builds upon the Euclidean analyses of \cite{latafat2024adaptive,latafat2024convergence}, and more generally follows the ``self-adaptive'' rationale of generating stepsizes solely based on past available data, without resorting to inner loops or requiring existence (or knowledge thereof) of any global modulus.
			However, the involvement of Bregman geometry brings forth several challenges that do not allow for straightforward extensions of these works.

			Each iteration revolves around three local estimates: two are Lipschitz-like estimates for \(\nabla f\) and for the \emph{forward operator}
			\begin{equation}\label{eq:Hk}
				H_k\coloneqq\nabla\kernel-\gamma_k\nabla f,
			\end{equation}
			and one measuring the gap between \(\D_{\kernel}(x,y)\) and \(\D_{\kernel}(y,x)\) at specific points.
			This latter measure is superfluous in Euclidean analyses, since the quadratic function \(\kernel=\j\) enjoys complete symmetry.

			\begin{subequations}
				\paragraph{Differentiable function \(f\)}
					First, based on \eqref{eq:Df<D},
					\begin{equation}\label{eq:lk}
						\ell_k
					\coloneqq
						\frac{\DD_f(x^k,x^{k-1})}{\DD_{\kernel}(x^k,x^{k-1})}
					\end{equation}
					provides an estimate of the relative smoothness modulus on the line segment between the last two consecutive iterates \(x^{k-1}\) and \(x^k\).
					This is the obvious counterpart of (the inverse of) a Barzilai-Borwein stepsize \cite{barzilai1988two}, and the local Lipschitz estimate of \(\nabla f\) used in \cite{latafat2024adaptive} for the Euclidean case, which indeed matches \eqref{eq:lk} when \(\kernel=\j\).

				\paragraph{Forward operator \(H_k\)}
					Inferring a Bregman equivalent of a local Lipschitz estimate of the forward operator \eqref{eq:Hk} is not as immediate.
					Indeed, replacing \(\norm{H_k(x^k)-H_k(x^{k-1})}^2\) with, say, a Bregman term \(\D_{\kernel^*}(H_k(x^k),H_k(x^{k-1}))\) does not seem to lead to quantities that naturally arise in the analysis.
					Our solution is more convoluted, and specifically given by
					\begin{equation}\label{eq:Lamk}
						\Lambda_{k,\delta}
					\coloneqq
						\frac{
							2\D_{\kernel^*}\bigl(\nabla\kernel(x^k)+\delta\bigl[H_k(x^k)-H_k(x^{k-1})\bigr],\nabla\kernel(x^k)\bigr)
						}{
							\delta^2\DD_{\kernel}(x^k,x^{k-1})
						}
					\end{equation}
					depending on some parameter \(\delta>0\) (specified later).
					Despite its deceptive intricacy, when \(\kernel=\j\) is quadratic, and thus so is its conjugate \(\kernel^*\), this estimate recovers the (square) Lipschitz estimate \(\tfrac{\norm{H_k(x^k)-H_k(x^{k-1})}^2}{\norm{x^k-x^{k-1}}^2}\)
					of \cite[Lem. 2.1(ii)]{latafat2024adaptive}, independently of the parameter \(\delta\).
					However, a judicious choice of \(\delta\) will be crucial for our convergence analysis in the generality of \cref{ass:basic}.
					The expression \eqref{eq:Lamk} for more general \(\kernel\) owes to the following Bregman extension of Young's inequality
					\[
						\innprod{x-y}{v}
					\leq
						\tfrac{1}{\delta}
						\D_{\kernel}(x,y)
						+
						\tfrac{1}{\delta}
						\D_{\kernel^*}\bigl(\nabla\kernel(y)+\delta v,\nabla\kernel(y)\bigr),
					\]
					see \cref{thm:BY} for a precise statement.

				\paragraph{Bregman kernel \(\kernel\)}
					Lastly, a local \emph{symmetry coefficient}
					\begin{equation}\label{eq:ak}
						\alpha_k
					\coloneqq
						\frac{\D_{\kernel}(x^k,x^{k-1})}{\D_{\kernel}(x^{k-1},x^k)}
					\in
						(0,\infty)
					\end{equation}
					allows us to express
			\end{subequations}
			\begin{equation}
				\DD_{\kernel}(x^k,x^{k-1})
			=
				\tfrac{1+\alpha_k}{\alpha_k}\D_{\kernel}(x^k,x^{k-1}).
			\end{equation}
			Note that \(\alpha_k>0\) holds for any \(k\) by essential strict convexity of \(\kernel\), regardless of whether or not \(\kernel\) has a \emph{global} (strictly positive) \emph{symmetry coefficient}
			\begin{equation}\label{eq:alpha}
				\alpha(\kernel)
			\coloneqq
				\inf_{\substack{x,y\in\interior\dom\kernel\\x\neq y}}\frac{\D_{\kernel}(x,y)}{\D_{\kernel}(y,x)}.
			\end{equation}
			Even when it does, a symmetry coefficient based on the global landscape of \(\kernel\) may be excessively conservative; instead, the use of local estimates enables the adoption of tighter constants, ultimately leading to larger stepsizes for Bregman proximal gradient iterations \eqref{eq:BPG}.
			The interested reader is referred to the recent work \cite{nilsson2025symmetry} for an in-depth analysis of the coefficient \(\alpha(\kernel)\) for a certain class of Bregman kernels \(\kernel\).

		%% ░░░░ 2.2 Proposed algorithms and main results ░░░░░░░░░░░░░░░░░░░░░░░░░░░░░░░░░░░░░░░░░░░░░░░░░░░░░░░░░░░░░░░░░
		\subsection{Proposed algorithms and main results}\label{sec:results}%
			Based on these three quantities, choose a stepsize \(\gamma_{k+1}\) and proceed with a Bregman proximal gradient update \eqref{eq:BPG}.
			We propose the following two options, where we let
			\[
				\rho_{k+1}
			\coloneqq
				\tfrac{\gamma_{k+1}}{\gamma_k}
			\]
			denote the ratio of consecutive stepsizes (so that \(\gamma_{k+1}=\rho_{k+1}\gamma_k\)):

			\medskip
			\par\noindent\phantomsection \label{alg:BY}%
				\parbox[t]{0.15\linewidth}{\bf\BY}%
				\hfill
				\fbox{\parbox[t]{0.82\linewidth}{%
					set \(\hat\rho_{k+1}=\sqrt{1+\rho_k}\) and \(\delta=2\hat\rho_{k+1}\), and update
					\begin{equation}\label{eq:BY:rhok*}
						\rho_{k+1}
					=
						\min\left\{
							\hat\rho_{k+1}
						,
							\frac{\alpha_k}{1+\alpha_k}
							\frac{
								1
							}{
								2\hat\rho_{k+1}
								\left[
									\Lambda_{k,\delta}
									-
									(1-\gamma_k\ell_k)
								\right]_+
							}
						\right\}
					\end{equation}
				}}%

			\medskip
			\noindent
			and, in case \(\kernel\) enjoys a symmetry coefficient \(\alpha=\alpha(\kernel)>0\),
			\medskip

			\par\noindent\phantomsection \label{alg:FNEa}%
				\parbox[t]{0.15\linewidth}{%
					{\bf\FNEa}

					\footnotesize
					(if \(\alpha(\kernel)>0\))%
				}
				\hfill
				\fbox{\parbox[t]{0.82\linewidth}{%
					set \(\hat\rho_{k+1}=\sqrt{\frac{1+\alpha}{2}+\rho_k}\) and \(\delta=\frac{2}{1+\alpha}\hat\rho_{k+1}\), and update
					\begin{equation}\label{eq:FNEa:rhok*}
						\rho_{k+1}
					=
						\min\left\{
							\hat\rho_{k+1}
						,
							\frac{\alpha}{
								2\hat\rho_{k+1}
								\left[
									\Lambda_{k,\delta}
									-
									(1-\gamma_k\ell_k)
								\right]_+
							}
						\right\}.
					\end{equation}
				}}%

			\bigskip
			\begin{remark}\label{rem:1/0}%
				We use the convention that \(\frac{1}{0}=\infty\), and remind that \([t]_+=\max\{0,t\}\).
				In particular, whenever \(\Lambda_{k,\delta}\leq1-\gamma_k\ell_k\) all updates reduce to \(\rho_{k+1}=\hat\rho_{k+1}\).
				It is implied that a starting point \(x^0\in C\) should be provided, as well as two stepsizes \(\gamma_0,\gamma_1>0\) for the first two iterations.
				We refer the reader to \cref{sec:simulation:BadaPG} for a practical initialization strategy.
				\hfill\qedsymbol
			\end{remark}

			The following theorem collects the main results for Bregman proximal gradient iterations \eqref{eq:BPG} with stepsizes selected according to the above rules.

			\begin{center}
			\fbox{\parbox{0.975\linewidth}{%
				\begin{theorem}[summary of main results]\label{thm:main}%
					Suppose that \cref{ass:basic} holds, and consider the iterates generated by \refBY{}.
					Then, one always has that
					\begin{equation}\label{eq:main:inf}
						\smash{
							\inf_{k\in\N}\varphi(x^k)=\inf_{\overline{C}}\varphi.
						}
					\end{equation}
					Moreover,
					\begin{enumerate}
					\item \label{thm:main:C}%
						If \(C\cap\argmin_{\overline{C}}\varphi\neq \emptyset\) (equivalently, if \(C\cap\argmin\varphi\neq\emptyset\)), then there exists \(x^\star\in C\cap \argmin\varphi\) such that \(x^k\to x^\star\).
					\item \label{thm:main:zone}%
						If \cref{ass:zone} holds, then \((x^k)_{k\in\N}\) is bounded and admits a unique optimal limit point.
					\end{enumerate}
					When \(\kernel\) has symmetry coefficient \(\alpha>0\), the same is true for \refFNEa.
				\end{theorem}%
			}}%
			\end{center}

			Any Legendre kernel \(\kernel\) with a non-open domain necessarily satisfies \(\alpha(\kernel)=0\) \cite[Prop. 2]{bauschke2017descent}.
			Consequently, the statement of \cref{thm:main:zone} is only nontrivial for the \refBY{} variant: a positive symmetry coefficient \(\alpha(\kernel)>0\), combined with the domain closedness required by \cref{ass:zone}, forces \(\dom\kernel=\R^n\).
			In this case, the stronger conclusion of \cref{thm:main:C} applies instead.

			Although the general theoretical results are weaker than in the Euclidean setting, our proposed methods demonstrate significant practical advantages.
			The numerical simulations in \cref{sec:simulations} confirm that the adaptive choices \refBY{} and \refFNEa{} enable larger stepsizes (even on average) and exhibit faster convergence, outperforming even aggressive linesearch strategies.
			Providing a firm theoretical basis for this observed performance, similar to what has been established in the Euclidean case, remains a compelling direction for future research.

			\begin{remark}\label{rem:eps}%
				Concerning \cref{thm:main:zone}, it can be shown that the entire sequence converges up to multiplying the right-hand side of \eqref{eq:BY:rhok*} by \((1-\epsilon)\) for some \(0<\epsilon\ll1\).
				Under \cref{ass:zone}, this slight modification of the stepsize rule guarantees that the sequences \(\bigl(\gamma_k(\varphi(x^k) - \inf_{\overline{C}}\varphi)\bigr)_{k\in\N}\) and \(\bigl(\D_\kernel(x^k,x^{k-1})\bigr)_{k\in\N}\) both converge to zero.
				This fact follows from a simple telescoping argument on \eqref{eq:UkBY:leq}, from which sequential convergence can be established arguing as in the proof of \cref{thm:xinC}; see \cref{proof:rem:eps} for the details.
				We conjecture that the same convergence result holds without this modification of the stepsize update rule, but a formal proof of this fact remains an open problem.
				\hfill\qedsymbol
			\end{remark}

			\subsubsection{Comparison with Euclidean methods}
				The \refBY{} variant can be interpreted as a Bregman analogue of the \adaPG{} update of \cite{latafat2024adaptive}, which reads
				\[
					\rho_{k+1}^{\text{\adaPG}}
				=
					\min\left\{
						\hat\rho_{k+1}
					,~
						\tfrac{1}{
							2\sqrt{\left[
								\Lambda_k
								-
								(1-\gamma_k\ell_k)
							\right]_+}
						}
					\right\}.
				\]
				Here, \(\ell_k\) and \(\Lambda_k\) are as in \eqref{eq:lk} and \eqref{eq:Lamk} with \(\kernel=\j\) (as already mentioned, in this case the latter is independent of the parameter \(\delta\) and is thus omitted from the subscript).
				On the other hand, with \(\kernel=\j\) (hence \(\alpha_k = 1\)) \eqref{eq:BY:rhok*} reads
				\begin{align*}
					\rho_{k+1}^{\text{\BY}}
				={} &
					\min\left\{
						\hat\rho_{k+1}
					,~
						\tfrac{
							1
						}{
							4\hat\rho_{k+1}
							\left[
								\Lambda_k
								-
								(1-\gamma_k\ell_k)
							\right]_+
						}
					\right\}
				\\
				\leq{} &
					\min\left\{
						\hat\rho_{k+1}
					,~
						\tfrac{
							1
						}{
							2\sqrt{
							\left[
								\Lambda_k
								-
								(1-\gamma_k\ell_k)
							\right]_+}
						}
					\right\}
				=
					\rho_{k+1}^{\text{\adaPG}},
				\end{align*}
				thus introducing a slight conservatism over the Euclidean predecessor; the inequality owes to the fact that \(\rho_{k+1}\leq\hat\rho_{k+1}\), hence that \(\rho_{k+1}^2\leq\rho_{k+1}\hat\rho_{k+1}\).

				Similarly, the \FNEa{} variant with \(\kernel=\j\) (hence \(\alpha=1\)) simplifies to a slightly more conservative variant of the update
				\[
					\rho_{k+1}^{\text{\adaPG}^{1,\frac{1}{2}}}
				\coloneqq
					\min\left\{
						\sqrt{1+\rho_k}
					,
						\tfrac{1}{
							\sqrt{
								2
								\left[
									\Lambda_{k,\delta}
									-
									(1-\gamma_k\ell_k)
								\right]_+
							}
						}
					\right\}
				\]
				of \cite[\adaPG$^{1,\frac{1}{2}}$]{latafat2024convergence}, having
				\[
					\rho_{k+1}^{\text{\FNEa}}
				=
					\min\left\{
						\hat\rho_{k+1}
					,
						\tfrac{1}{
							2\hat\rho_{k+1}
							\left[
								\Lambda_{k,\delta}
								-
								(1-\gamma_k\ell_k)
							\right]_+
						}
					\right\}
				\leq
					\rho_{k+1}^{\text{\adaPG}^{1,\frac{1}{2}}}
				\]
				(in all occurrences throughout this subsubsection, \(\hat\rho_{k+1}=\sqrt{1+\rho_k}\)).

				\begin{remark}[quadratic kernels]%
					As explained in \cref{sec:estimates}, for general \(\kernel\) the curvature estimate \(\Lambda_{k,\delta}\) as in \eqref{eq:Lamk} depends on the Bregman--Young parameter \(\delta>0\).
					In exisiting Euclidean analyses, this parameter is optimally chosen as a suitable multiple of the ratio \(\rho_{k+1}=\gamma_{k+1}/\gamma_k\), a feasible choice given that the value of \(\Lambda_{k,\delta}\) is independent of \(\delta\).
					This is not the case for more general kernels \(\kernel\), whence the above-discussed conservatism originates: the value of \(\gamma_{k+1}\) depends on \(\Lambda_{k,\delta}\), and should \(\Lambda_{k,\delta}\) in turn depend on \(\gamma_{k+1}\) a circular dependency would arise.

					Specializing \refFNEa{} to quadratic \(\kernel=\tfrac{1}{2}\norm{{}\cdot{}}_Q^2\) with \(Q\in\symm_{++}(\R^n)\), this issue does not persist and the tighter analyses of the Euclidean cases are recovered.
					The corresponding algorithm produces iterates
					\begin{gather*}
						x^{k+1}
					=
						\argmin_{x\in\R^n}\left\{
							\innprod{\nabla f(x^k)}{x-x^k}
							+
							g(x)
							+
							\tfrac{1}{2\gamma_{k+1}}\norm{x-x^k}_Q^2
						\right\}
					\shortintertext{with stepsizes chosen as}
						\gamma_{k+1}
					=
						\gamma_k\min\biggl\{
							\sqrt{1+\tfrac{\gamma_k}{\gamma_{k-1}}}
						,
							\tfrac{1}{
								\sqrt{
									2
									\left[
										\gamma_k^2 L_k^2-\gamma_k\ell_k
									\right]_+
								}
							}
						\biggr\},
					\shortintertext{where}
						\ell_k
					=
						\tfrac{
							\innprod{\nabla f(x^k)-\nabla f(x^{k-1})}{x^k-x^{k-1}}
						}{
							\norm{x^k-x^{k-1}}_Q^2
						}
					\quad\text{and}\quad
						L_k
					=
						\tfrac{
							\norm{\nabla f(x^k)-\nabla f(x^{k-1})}_{Q^{-1}}
						}{
							\norm{x^k-x^{k-1}}_Q
						}.
					\end{gather*}
					This expression follows from the easily verifiable fact that \(\Lambda_{k,\delta}=\gamma_k^2L_k^2-2\gamma_k\ell_k+1\) (independently of \(\delta\)) in this case.
				\hfill\qedsymbol
				\end{remark}

	%% ██ 3. Main inequalities ███████████████████████████████████████████████████████████████████████████████████████████
	\section{Main inequalities}
		To ease the subsequent discussion, we introduce some convenient notational shorthands and remind some of those already encountered.
		Relative to the iterates generated by \eqref{eq:BPG}, for any \(k\in\N\) and \(x\in\dom\varphi\) we denote
		\begin{align}
		\label{eq:Pk}
			P_k(x)
		\coloneqq{} &
			\varphi(x^k)-\varphi(x)
		\shortintertext{and}
		\label{eq:Bk}
			B_{k+1}
		\coloneqq{} &
			\rho_{k+1}
			\innprod{x^{k+1}-x^k}{H_k(x^k)-H_k(x^{k-1})},
		\end{align}
		where \(\rho_{k+1}=\tfrac{\gamma_{k+1}}{\gamma_k}\) and
		\begin{equation}\label{eq:H:def}
			H_k\coloneqq\nabla\kernel-\gamma_k\nabla f
		\end{equation}
		is the ``forward'' operator.
		Due to convexity of \(g\), the BPG updates in \eqref{eq:BPG} are characterized by the subgradient inclusion
		\begin{subequations}
			\begin{align}
			\label{eq:subgrad}
				\tilde{\nabla}g(x^{k+1})
			\coloneqq{} &
				\tfrac{\nabla\kernel(x^k)-\nabla\kernel(x^{k+1})}{\gamma_{k+1}}
				-
				\nabla f(x^k)
			\in
				\partial g(x^{k+1}).
			\intertext{%
				Throughout, we use the notation \(\tilde{\nabla}g(x^{k+1})\) to indicate this particular element of the subgradient of \(g\) along the iterates.
				Similarly, we use
			}
			\nonumber
				\tilde{\nabla}\varphi(x^{k+1})
			\coloneqq{} &
				\tilde{\nabla} g(x^{k+1})
				+
				\nabla f(x^{k+1})
			\\
			\label{eq:subgradvarphi}
			={} &
				\tfrac{H_{k+1}(x^k)-H_{k+1}(x^{k+1})}{\gamma_{k+1}}
			\in
				\partial\varphi(x^{k+1}).
			\end{align}
		\end{subequations}
		In light of these, we adopt the notation
		\begin{gather*}
			\tD_g(w,x^k)
		\coloneqq
			g(w)-g(x^k)-\innprod{\tilde{\nabla}g(x^k)}{w-x^k},
		\shortintertext{and similarly}
			\tD_{\varphi}(w,x^k)
		\coloneqq
			{\varphi}(w)-{\varphi}(x^k)-\innprod{\tilde{\nabla}{\varphi}(x^k)}{w-x^k},
		\end{gather*}
		which are both positive quantities for any \(k\in\N\) and \(w\in\R^n\).

		The main identity in our study is an adaptation to the Bregman setting of the inequality in \cite[Lem. 2.2]{latafat2024adaptive}.
		The proof closely patterns the one in the reference, and is provided in the appendix for completeness.

		\begin{lemma}[main identity]\label{thm:eq}%
			Suppose that \cref{ass:basic} holds, and starting from \(x^0\in C\) consider Bregman proximal gradient iterations \eqref{eq:BPG} with stepsizes \(\gamma_k>0\).
			Then, for any \(x\in\dom\varphi\cap\dom\kernel\), \(\vartheta_{k+1}\geq0\), \(k\in\N\), it holds that
			\begin{align*}
			&
				\D_{\kernel}(x,x^{k+1})
				+
				\gamma_{k+1}(1+\vartheta_{k+1})P_k(x)
				+
				\D_{\kernel}(x^{k+1},x^k)
			\\
			={} &
				\D_{\kernel}(x,x^k)
				+
				\gamma_{k+1}\vartheta_{k+1} P_{k-1}(x)
				-
				\rho_{k+1}\vartheta_{k+1}(1-\gamma_k\ell_k)
				\DD_{\kernel}(x^k,x^{k-1})
				+
				B_{k+1}
			\\
			&
				-
				\gamma_{k+1}
				\Bigl\{
					\D_f(x,x^k)
					+
					\tD_g(x^{k+1},x^k)
					+
					\tD_g(x,x^{k+1})
					+
					\vartheta_{k+1}\tD_\varphi(x^{k-1},x^k)
				\Bigr\},
			\end{align*}
			where \(P_k(x)\), \(B_{k+1}\), and \(\ell_k\) are as in \eqref{eq:Pk}, \eqref{eq:Bk}, and \eqref{eq:lk}.
			In particular, denoting
			\[
				\widehat{\U}_k(x)
			\coloneqq
				\D_{\kernel}(x,x^k)
				+
				\gamma_k(1+\vartheta_k)P_{k-1}(x)
				+
				\D_{\kernel}(x^k,x^{k-1}),
			\]
			one has that
			\begin{align}
			\nonumber
				\widehat{\U}_{k+1}(x)
			\leq{} &
				\widehat{\U}_k(x)
				-
				\gamma_k(1+\vartheta_k-\rho_{k+1}\vartheta_{k+1})P_{k-1}(x)
				+
				B_{k+1}
			\\
			&
				-
				\left[
					\tfrac{\alpha_k}{1+\alpha_k}
					+
					\rho_{k+1}\vartheta_{k+1}(1-\gamma_k\ell_k)
				\right]
				\DD_{\kernel}(x^k,x^{k-1}).
			\label{eq:ineq0}
			\end{align}
		\end{lemma}
		\begin{proof}
			See \cref{proof:thm:eq}.
		\end{proof}

		The following two subsections will be devoted to developing analogues of Young's inequality allowing us to bound the inner product term \(B_{k+1}\) in terms of Bregman distances.
		These pinpoint the main departure from previous Euclidean analyses, in particular the need to introduce a new parameter \(\hat\rho_{k+1}\) that generates some unavoidable conservatism; see the discussion after \cref{thm:main}.

		%% ░░░░ 3.1 Young's inequality in the Bregman sense ░░░░░░░░░░░░░░░░░░░░░░░░░░░░░░░░░░░░░░░░░░░░░░░░░░░░░░░░░░░░░
		\subsection{Young's inequality in the Bregman sense}
			Young's inequality is a very simple but powerful tool enabling to bound inner products in terms of sum of squares.
			Its derivation is elementary, all revolving around the fact that, for any \(u,v\in\R^n\) and \(\delta>0\),
			\[
				\innprod{u}{v}
			=
				\tfrac{1}{\delta}\innprod{u}{\delta v}
			=
				\tfrac{1}{2\delta}\norm{u}^2
				+
				\tfrac{\delta}{2}\norm{v}^2
				-
				\tfrac{1}{2\delta}\norm{u-\delta v}^2.
			\]
			Discarding the negative term leaves us with the familiar bound holding for any \(\delta>0\).
			The same arguments can be extended beyond quadratic norms to more general Bregman distances by means of the three-point identity of \cref{thm:3p}.

			\begin{lemma}[Young's inequality in the Bregman sense]\label{thm:BY}%
				Let \(\kernel:\R^n\to\Rinf\) be Legendre (not necessarily 1-coercive).
				Then, for any \(x\in\dom\kernel\), \(y\in\interior\dom\kernel\), \(v\in\R^n\), and \(\delta>0\) one has
				\begin{equation}\label{eq:BY}
					\innprod{x-y}{v}
				\leq
					\tfrac{1}{\delta}
					\D_{\kernel}(x,y)
					+
					\tfrac{1}{\delta}
					\D_{\kernel^*}\bigl(\nabla\kernel(y)+\delta v,\nabla\kernel(y)\bigr).
				\end{equation}
			\end{lemma}
			\begin{proof}
				Without loss of generality we may assume that \(\nabla\kernel(y)+\delta v\in\dom\kernel^*\), for otherwise the right-hand side of \eqref{eq:BY} is infinite and the inequality trivializes.
				Since \(\nabla\kernel(y)\in\interior\dom\kernel^*\) by \cref{fact:D:Dphi*}, for any \(t\in(0,1)\) one has that
				\(
					\nabla\kernel(y)+t\delta v\in\interior\dom\kernel^*
				\)
				by \cite[Prop. 3.44]{bauschke2017convex}.
				We have
				\begin{align*}
					\innprod{x-y}{v}
				={} &
					\tfrac{1}{t\delta}
					\innprod{x-y}{t\delta v}
				\\
				={} &
					\tfrac{1}{t\delta}
					\innprod{x-y}{\nabla\kernel(\nabla\kernel^*\bigl(\nabla\kernel(y)+t\delta v)\bigr)-\nabla\kernel(y)}
				\\
				={} &
					\tfrac{1}{t\delta}
					\D_{\kernel}(x,y)
					+
					\tfrac{1}{t\delta}
					\D_{\kernel}\bigl(y,\nabla\kernel^*\bigl(\nabla\kernel(y)+t\delta v\bigr)\bigr)
				\\
				&
					-
					\tfrac{1}{t\delta}
					\D_{\kernel}\bigl(x,\nabla\kernel^*\bigl(\nabla\kernel(y)+t\delta v\bigr)\bigr)
				\\
				\leq{} &
					\tfrac{1}{t\delta}
					\D_{\kernel}(x,y)
					+
					\tfrac{1}{t\delta}
					\D_{\kernel^*}\bigl(\nabla\kernel(y)+t\delta v,\nabla\kernel(y)\bigr).
			\end{align*}
				The second equality follows from Legendreness of \(\kernel\) via \cref{fact:D:Dphi*}, and the third one is the three point identity of \cref{thm:3p}.
				The inequality follows from \cref{fact:D:DtoD*} and the fact that \(\D_\kernel\geq0\).
				By taking the limit as \(t\to1^-\) and appealing to \cite[Thm. 7.5]{rockafellar1970convex}, the claimed inequality \eqref{eq:BY} is obtained.
			\end{proof}

			When \(\kernel=\j\), one recovers the usual Young's inequality
			\[
				\innprod{x-y}{v}
			\leq
				\tfrac{1}{2\delta}\norm{x-y}^2
				+
				\tfrac{\delta}{2}\norm{v}^2
			\eqqcolon
				\psi_{\j}(\delta).
			\]
			Note that the right-hand side \(\psi_{\j}(\delta)\) diverges as \(\delta\to0^+\) and \(\delta\to\infty\), and attains a unique minimizer at \(\delta=\frac{\norm{x-y}}{\norm{v}}\).
			This specific choice of \(\delta\) leads to the Cauchy-Schwarz inequality
			\[
				\innprod{x-y}{v}
			\leq
				\inf_{\delta>0}\psi_{\j}(\delta)
			=
				\norm{x-y}\norm{v}.
			\]
			A similar pattern occurs for the right-hand side in \eqref{eq:BY}, although with some complications arising because of the dependency on \(\delta\) for the second Bregman distance.

			\begin{lemma}[A Cauchy--Schwarz inequality in the Bregman sense]\label{thm:BCS}%
				In the setting of \cref{thm:BY}, suppose that \(\kernel\) is 1-coercive. Then, either the upper bound in \eqref{eq:BY} is always decreasing in \(\delta\), or it is minimized at a unique \(\delta^\star>0\) which is characterized by the identity
				\begin{equation}\label{eq:delta*}
					\D_{\kernel^*}\bigl(\nabla\kernel(y),\nabla\kernel(y)+\delta^\star v\bigr)
				=
					\D_{\kernel}(x,y).
				\end{equation}
				In this latter case, which is necessarily true when \(\dom\kernel\) is open, one has that
				\begin{equation}\label{eq:BCS}
					\innprod{x-y}{v}
				\leq
					\tfrac{1}{\delta^\star}
					\DD_{\kernel^*}\bigl(\nabla\kernel(y)+\delta^\star v,\nabla\kernel(y)\bigr).
				\end{equation}
			\end{lemma}
			\begin{proof}
				Let
				\[
					\psi_{\kernel}(\delta)
				\coloneqq
					\frac{
						\D_{\kernel}(x,y)
						+
						\D_{\kernel^*}\bigl(\nabla\kernel(y)+\delta v,\nabla\kernel(y)\bigr)
					}{\delta}
				\]
				denote the right-hand side of the Bregman--Young inequality \eqref{eq:BY}.
				A simple computation reveals that its derivative is
				\begin{align*}
					\psi_{\kernel}'(\delta)
				={} &
					\frac{
						\innprod{\nabla\kernel^*\bigl(\nabla\kernel(y)+\delta v\bigr)-y}{\delta v}
						-
						\D_{\kernel}(x,y)
						-
						\D_{\kernel^*}\bigl(\nabla\kernel(y)+\delta v,\nabla\kernel(y)\bigr)
					}{\delta^2}
				\\
				={} &
					\frac{
						\D_{\kernel^*}\bigl(\nabla\kernel(y),\nabla\kernel(y)+\delta v\bigr)
						-
						\D_{\kernel}(x,y)
					}{\delta^2}.
				\end{align*}
				Since \(\kernel^*\) is strictly convex, the numerator is strictly increasing. Moreover, as \(\D_{\kernel}(x,y)>0\), it is strictly negative at \(\delta=0\).
				Therefore, it is either negative for all \(\delta>0\) or it vanishes at a unique \(\delta^\star\) as in the statement, which must be the global minimum of \(\psi_{\kernel}\).

				Thus, in this latter case,
				\begin{equation}\label{eq:infpsi}
					\innprod{x-y}{v}
				\leq
					\inf_{\delta>0}\psi_{\kernel}(\delta)
				=
					\frac{
						\D_{\kernel}(x,y)
						+
						\D_{\kernel^*}\bigl(\nabla\kernel(y)+\delta^\star v,\nabla\kernel(y)\bigr)
					}{\delta^\star},
				\end{equation}
				which by \eqref{eq:delta*} expands to the right-hand side of \eqref{eq:BCS}.

				Finally, if \(\dom\kernel\) is open it follows from \cite[Cor. 3.11]{bauschke1997legendre} that \(\D_{\kernel^*}\bigl(\nabla\kernel(y),\cdot\bigr)\) is coercive, and thus the numerator in the expression of \(\psi_{\kernel}'(\delta)\) cannot be negative for all \(\delta\).
			\end{proof}

			Note that the right-hand side of \eqref{eq:BCS} does depend on \(x\) via the parameter \(\delta^\star\), as evident by its definition \eqref{eq:delta*}.
			This upper bound can be relaxed into a simplified form whenever \(\kernel\) has a (strictly positive) \emph{symmetry coefficient} \(\alpha=\alpha(\kernel)>0\) as in \eqref{eq:alpha}.
			In this case, \(\dom\kernel\) must be open \cite[Prop. 2]{bauschke2017descent} and, since \(\alpha(\kernel)=\alpha(\kernel^*)\) \cite[Rem. 2(b)]{bauschke2017descent}, \eqref{eq:infpsi} can be further expanded into
			\[
				\innprod{x-y}{v}
			\leq
				\frac{
					\D_{\kernel}(x,y)
					+
					\frac{1}{\alpha}\D_{\kernel^*}\bigl(\nabla\kernel(y),\nabla\kernel(y)+\delta^\star v\bigr)
				}{\delta^\star},
			\]
			which combined with \eqref{eq:delta*} results in the following simplified version.

			\begin{corollary}
				Consider the setting of \cref{thm:BCS} with \(\innprod{x-y}{v}>0\), and suppose that \(\kernel\) has symmetry coefficient \(\alpha>0\).
				Then, \(\delta^\star\) as in \eqref{eq:delta*} exists, and one has
				\begin{equation}\label{eq:BCS'}
					\innprod{x-y}{v}
				\leq
					\tfrac{1+\alpha}{\alpha}
					\tfrac{1}{\delta^\star}
					\D_{\kernel}(x,y).
				\end{equation}
			\end{corollary}

		%% ░░░░ 3.2 Bounding the inner product B_{k+1} ░░░░░░░░░░░░░░░░░░░░░░░░░░░░░░░░░░░░░░░░░░░░░░░░░░░░░░░░░░░░░░░░░░░
		\subsection[Bounding the inner product]{Bounding the inner product \(B_{k+1}\)}\label{sec:Bk}%
			Patterning previous analyses of adaptive stepsizes, our goal is to turn the identity of \cref{thm:eq} into a descent-type inequality on some merit function.
			The bottleneck lies in the inner product term \(B_{k+1}\) as in \eqref{eq:Bk}, which in previous analyses restricted to the Euclidean case was handled via standard Young's or Cauchy--Schwarz bounds.

			The Bregman version of Young's inequality given in \cref{thm:BY} allows us to replicate these ideas, but with an important caveat.
			Indeed, the presence of \(\delta\) within the argument of the Bregman distance \(\D_{\kernel^*}\) constrains our choice on the parameter, causing a slight departure from the easier Euclidean case in which such complication does not arise.
			In this subsection, we identify two possible options, each leading to one of the two algorithmic variants \refBY{} and \refFNEa{}.

			\subsubsection{\refBY{} bound}\label{sec:Bk*:BY}%
				A direct application of Young's inequality of \cref{thm:BY} allows us to bound the inner product \(B_{k+1}\) as
				\begin{align}
				\nonumber
					B_{k+1}
				\leq{} &
					\tfrac{\rho_{k+1}}{\delta_{k+1}}\D_{\kernel}(x^{k+1},x^k)
				\\
				\nonumber
				&
					+
					\tfrac{\rho_{k+1}}{\delta_{k+1}}\D_{\kernel^*}\bigl(\nabla\kernel(x^k)+\delta_{k+1}\bigl[H_k(x^k)-H_k(x^{k-1})\bigr],\nabla\kernel(x^k)\bigl)
				\shortintertext{%
					for any \(\delta_{k+1}>0\), which in terms of the Lipschitz-like estimate \(\Lambda_{k,\delta}\) as in \eqref{eq:Lamk} reads
				}
				={} &
					\tfrac{\rho_{k+1}}{\delta_{k+1}}\D_{\kernel}(x^{k+1},x^k)
					+
					\tfrac{\delta_{k+1}\rho_{k+1}}{2}
					\Lambda_{k,\delta_{k+1}}
					\DD_{\kernel}(x^k,x^{k-1}).
				\label{eq:Bk:BY}
				\end{align}
				As we will see, the employment of this inequality combined with a specific choice of \(\delta_{k+1}\) will lead to the stepsize update in the \refBY{} variant.

			\subsubsection{\refFNEa{} bound}\label{sec:Bk*:FNEa}%
				The bound leading to the \refFNEa{} variant follows from the combination of \eqref{eq:Bk:BY} and the following lemma, which furnishes a Bregman generalization of firm nonexpansiveness in the adaptive setting.

				\begin{lemma}\label{thm:FNE}
					Let \(\kernel:\R^n\to\Rinf\) be Legendre, and \(f,g:\R^n\to\Rinf\) be proper, lsc, and convex, with \(f\) differentiable on \(C\coloneqq\interior\dom\kernel\) and \(\dom g\cap C\neq\emptyset\).
					Then, denoting \(H_k\coloneqq\nabla\kernel-\gamma_k\nabla f\), Bregman proximal gradient iterations \eqref{eq:BPG} with stepsizes \(\gamma_k>0\) and starting at some \(x^0\in C\) satisfy
					\[
						B_{k+1}
					\coloneqq
						\tfrac{\gamma_{k+1}}{\gamma_k}
						\innprod{x^{k+1}-x^k}{H_k(x^k)-H_k(x^{k-1})}
					\geq
						\DD_{\kernel}(x^{k+1},x^k)
					\quad
						\forall k\geq 1.
					\]
				\end{lemma}
				\begin{proof}
					Follows by observing that
					\begin{align*}
					&
						\tfrac{\gamma_{k+1}}{\gamma_k}
						\bigl(H_k(x^k)-H_k(x^{k-1})\bigr)
					\\
					={} &
						\nabla\kernel(x^{k+1})-\nabla\kernel(x^k)
						+
						\gamma_{k+1}
						\Bigl[
							\smash{
								\overbrace{
									\tfrac{H_{k+1}(x^k)-\nabla\kernel(x^{k+1})}{\gamma_{k+1}}
								}^{\tilde{\nabla}g(x^{k+1})}
								-
								\overbrace{
									\tfrac{H_k(x^{k-1})-\nabla\kernel(x^k)}{\gamma_k}
								}^{\tilde{\nabla}g(x^k)}
							}
						\Bigr],
					\end{align*}
					hence that
					\(
						B_{k+1}
					=
						\DD_{\kernel}(x^{k+1},x^k)
						+
						\gamma_{k+1}\tD_g(x^{k+1},x^k)
					\geq
						\DD_{\kernel}(x^{k+1},x^k)
					\).
				\end{proof}

				We may thus complement \eqref{eq:Bk:BY} with a lower bound as
				\begin{align*}
					\DD_{\kernel}(x^{k+1},x^k)
				\leq{} &
					B_{k+1}
				\\
				\leq{} &
					\tfrac{\rho_{k+1}}{\delta_{k+1}}\D_{\kernel}(x^{k+1},x^k)
					+
					\tfrac{\delta_{k+1}\rho_{k+1}}{2}
					\Lambda_{k,\delta_{k+1}}
					\DD_{\kernel}(x^k,x^{k-1}).
				\end{align*}
				If \(\kernel\) has symmetry coefficient \(\alpha>0\) as in \eqref{eq:alpha}, then note that
				\[
					\DD_{\kernel}(x^{k+1},x^k)
				\geq
					(1+\alpha)\D_{\kernel}(x^{k+1},x^k).
				\]
				Combined with the previous inequality, we obtain
				\[
					\left(
						1+\alpha-\tfrac{\rho_{k+1}}{\delta_{k+1}}
					\right)
					\D_{\kernel}(x^{k+1},x^k)
				\leq
					\tfrac{\delta_{k+1}\rho_{k+1}}{2}
					\Lambda_{k,\delta_{k+1}}
					\DD_{\kernel}(x^k,x^{k-1}),
				\]
				which plugged into \eqref{eq:Bk:BY} leads to
				\begin{equation}
				\label{eq:Bk:FNEa}
					B_{k+1}
				\leq
					\tfrac{
						(1+\alpha)\delta_{k+1}^2
					}{
						(1+\alpha)\delta_{k+1}-\rho_{k+1}
					}
					\tfrac{\rho_{k+1}}{2}
					\Lambda_{k,\delta_{k+1}}
					\DD_{\kernel}(x^k,x^{k-1}),
				\end{equation}
				holding for any \(\delta_{k+1}\) such that \((1+\alpha)\delta_{k+1}-\rho_{k+1}>0\).

		%% ░░░░ 3.3 A merit function for B-adaPG ░░░░░░░░░░░░░░░░░░░░░░░░░░░░░░░░░░░░░░░░░░░░░░░░░░░░░░░░░░░░░░░░░░░░░░░░░
		\subsection{A merit function for \refBY}\label{sec:BY:merit}%
			By bounding the inner product term \(B_{k+1}\) with \eqref{eq:Bk:BY}, the inequality \eqref{eq:ineq0} in \cref{thm:eq} reveals that Bregman proximal gradient iterations \eqref{eq:BPG} with arbitrary stepsizes \(\gamma_k>0\) satisfy
			\begin{align}
			\nonumber
			&
				\widehat{\U}_{k+1}(x)
				-
				\tfrac{\rho_{k+1}}{\delta_{k+1}}
				\D_{\kernel}(x^{k+1},x^k)
			\\
			\nonumber
			\leq{} &
				\widehat{\U}_k(x)
				-
				\tfrac{\rho_k}{\delta_k}
				\D_{\kernel}(x^k,x^{k-1})
				-
				\gamma_k(1+\vartheta_k-\rho_{k+1}\vartheta_{k+1})P_{k-1}(x)
			\\
			&
				-
				\left[
					\bigl(1-\tfrac{\rho_k}{\delta_k}\bigr)
					-
					\rho_{k+1}(1+\alpha_k)
					\tfrac{
						\frac{\delta_{k+1}}{2}
						\Lambda_{k,\delta_{k+1}}
						-
						\vartheta_{k+1}
						(1-\gamma_k\ell_k)
					}{\alpha_k}
				\right]
				\D_{\kernel}(x^k,x^{k-1})
			\label{eq:UkBYdt:leq}
			\end{align}
			for any \(x\in\dom\varphi\cap\dom\kernel\), \(\vartheta_k\geq0\), and \(\delta_k>0\), \(k\in\N\).
			Imposing that the terms in brackets multiplying \(P_{k-1}(x)\) and \(\D_{\kernel}(x^k,x^{k-1})\) in the right-hand side of \eqref{eq:UkBYdt:leq} are negative amounts to the following two conditions:
			\begin{subequations}\label{eq:gamk:ineq12}
				\begin{gather}
				\label{eq:gamk:ineq1}
					\vartheta_{k+1}\rho_{k+1}\leq 1+\vartheta_k
				\shortintertext{and}
				\label{eq:gamk:ineq2}
					1-\tfrac{\rho_k}{\delta_k}
				\geq
					\rho_{k+1}
					\tfrac{1+\alpha_k}{\alpha_k}
					\bigl[
						\tfrac{\delta_{k+1}}{2}
						\Lambda_{k,\delta_{k+1}}
						-
						\vartheta_{k+1}(1-\gamma_k\ell_k)
					\bigr].
				\end{gather}
			\end{subequations}
			In line with the analyses of \cite{malitsky2020adaptive,latafat2024adaptive}, a convenient choice for the parameter \(\delta_{k+1}\) is \(\delta_{k+1} = 2\rho_{k+1}\).
			However, this choice is not feasible in our more general setting.
			Specifically, selecting \(\rho_{k+1}\) to satisfy \eqref{eq:gamk:ineq2} requires knowledge of the quantity \(\Lambda_{k,\delta_{k+1}}\), which generally depends on \(\delta_{k+1}\) itself (except in the special case where \(\kernel\) is quadratic).
			As a result, setting \(\delta_{k+1}=2\rho_{k+1}\) would create a circular dependency between the two parameters.
			To complicate things further, note that the left-hand side of \eqref{eq:gamk:ineq2} indicates that a constraint \(\delta_k>\rho_k\) must be in place in order to ensure the existence of a \(\rho_{k+1}>0\) satisfying the inequality.

			In order to resolve this circular dependence, we introduce a parameter \(\hat\rho_{k+1}\) that shall provide an overestimate
			\begin{equation}\label{eq:rhokhatgeq}
				\hat\rho_{k+1}\geq\rho_{k+1}
			\end{equation}
			while only being based on information available at iteration \(k\).
			Its explicit value will be revealed shortly after.
			Then, we may conveniently select \(\delta_k=2\hat\rho_k\) and \(\vartheta_k=\hat\rho_k\) so that \eqref{eq:gamk:ineq12} simplifies as
			\[
				\rho_{k+1}\hat\rho_{k+1}
			\leq
				1+\hat\rho_k
			\quad\text{and}\quad
				1-\tfrac{\rho_k}{2\hat\rho_k}
			\geq
				\rho_{k+1}
				\hat\rho_{k+1}
				\tfrac{1+\alpha_k}{\alpha_k}
				\bigl[
					\Lambda_{k,\delta_{k+1}}
					-
					(1-\gamma_k\ell_k)
				\bigr],
			\]
			that is,
			\begin{equation}\label{eq:BY:rhok*+}
				\rho_{k+1}
			\leq
				\min\left\{
					\frac{1+\hat\rho_k}{\hat\rho_{k+1}}
					,\,
					\frac{\alpha_k}{1+\alpha_k}
					\frac{
						1+ \frac{\hat\rho_k - \rho_k}{\hat\rho_k}
					}{
						2\hat\rho_{k+1}
						\bigl[
							\Lambda_{k,2\hat\rho_{k+1}}
							-
							(1-\gamma_k\ell_k)
						\bigr]_+
					}
				\right\}.
			\end{equation}
			Due to \eqref{eq:rhokhatgeq} the term \(\frac{\hat\rho_k - \rho_k}{\hat\rho_k}\) in the numerator of the second update is always nonnegative.
			Using this, we show in the next lemma that the update rule of \refBY{} always complies with the above bound.
			In fact, while it is possible to retain the term \(\frac{\hat\rho_k - \rho_k}{\hat\rho_k}\) in the update of \refBY{}, for the sake of a neater expression we opted to omit it at the cost of introducing slight conservatism.

			\begin{lemma}\label{thm:BY:leq}%
				Suppose that \cref{ass:basic} holds, and consider the iterates generated by Bregman proximal gradient iterations \eqref{eq:BPG} with \(\gamma_k\) and \(\hat\rho_k\) selected according to \refBY.
				Then, both \eqref{eq:rhokhatgeq} and \eqref{eq:BY:rhok*+} are satisfied for any \(k\in\N\).

				Moreover, denoting
				\begin{equation}\label{eq:UkBY}
					\U_k(x)
				\coloneqq
					\D_{\kernel}(x,x^k)
					+
					\gamma_k(1+\hat\rho_k)P_{k-1}(x)
					+
					(1-\tfrac{\rho_k}{2\hat\rho_k})
					\D_{\kernel}(x^k,x^{k-1}),
				\end{equation}
				one has that
				\begin{equation}\label{eq:UkBY:leq}
					\U_{k+1}(x)
				\leq
					\U_k(x)
					-
					\gamma_k(
						\smash{
							\overbrace{1+\hat\rho_k-\rho_{k+1}\hat\rho_{k+1}}^{\geq\hat\rho_k-\rho_k\geq0}
						}
					)P_{k-1}(x)
				\end{equation}
				holds for any \(x\in\dom\varphi\cap\dom\kernel\) and \(k\in\N\).
			\end{lemma}
			\begin{proof}
				The update \eqref{eq:BY:rhok*} clearly ensures that \(\rho_{k+1}\leq\hat\rho_{k+1}\) always holds.
				Note that the second element in the minimum within \eqref{eq:BY:rhok*+} coincides with that in \eqref{eq:BY:rhok*};
				moreover, since \(\rho_{k+1}\leq\hat\rho_{k+1}\) one has that
				\(
					\rho_{k+1}\hat\rho_{k+1}
				\leq
					\hat\rho_{k+1}^2
				=
					1+\rho_k
				\leq
					1+\hat\rho_k
				\),
				altogether confirming that the validity of \eqref{eq:BY:rhok*+}.

				Observe further that the bound \(\hat\rho_k\geq\rho_k\) ensures that both elements in the minimum of \eqref{eq:BY:rhok*} are strictly positive, and thus so are the generated stepsizes \(\gamma_k\).
				Finally, \eqref{eq:UkBY:leq} follows from \eqref{eq:UkBYdt:leq} with the specified choices of \(\vartheta_k\) and \(\delta_k\).
			\end{proof}

		%% ░░░░ 3.4 A merit function for B-adaPG_a ░░░░░░░░░░░░░░░░░░░░░░░░░░░░░░░░░░░░░░░░░░░░░░░░░░░░░░░░░░░░░░░░░░░░░░░
		\subsection{A merit function for \refFNEa}
			In case \(\kernel\) has a symmetry coefficient \(\alpha>0\), we may leverage the bound \eqref{eq:Bk:FNEa} for \(B_{k+1}\).
			By doing so, the inequality \eqref{eq:ineq0} in \cref{thm:eq} becomes
			\begin{multline*}
				\widehat{\U}_{k+1}(x)
			\leq
				\widehat{\U}_k(x)
				-
				\gamma_k(1+\vartheta_k-\rho_{k+1}\vartheta_{k+1})P_{k-1}(x)
			\\
				-
				\left\{
					\tfrac{\alpha_k}{1+\alpha_k}
					-
					\rho_{k+1}
					\left[
						\tfrac{
							(1+\alpha)\delta_{k+1}^2
						}{
							(1+\alpha)\delta_{k+1}-\rho_{k+1}
						}
						\tfrac{1}{2}
						\Lambda_{k,\delta_{k+1}}
						-
						\vartheta_{k+1}(1-\gamma_k\ell_k)
					\right]
				\right\}
				\DD_{\kernel}(x^k,x^{k-1}).
			\end{multline*}

			Once again, we introduce a parameter \(\hat\rho_{k+1}\) (to be specified later) that satisfies \eqref{eq:rhokhatgeq} whilst based on information available at iteration \(k\).
			We can conveniently set \(\delta_{k+1}=\vartheta_{k+1}=\frac{2}{1+\alpha}\hat\rho_{k+1}\), and use the bound
			\(
				(1+\alpha)\delta_{k+1}-\rho_{k+1}
			\geq
				(1+\alpha)\delta_{k+1}-\hat\rho_{k+1}
			\)
			to simplify the coefficient of \(\Lambda_{k,\delta_{k+1}}\).
			Combined with the fact that
			\(
				\tfrac{\alpha_k}{1+\alpha_k}
			\geq
				\tfrac{\alpha}{1+\alpha}
			\),
			the inequality simplifies as
			\begin{align*}
				\U^{\alpha}_{k+1}(x)
			\coloneqq{} &
				\D_{\kernel}(x,x^{k+1})
				+
				\D_{\kernel}(x^{k+1},x^k)
				+
				\gamma_{k+1}\left(1+\tfrac{2}{1+\alpha}\hat\rho_{k+1}\right)P_k(x)
			\\
			\leq{} &
				\U^{\alpha}_k(x)
				-
				\gamma_k(1+\tfrac{2}{1+\alpha}\hat\rho_k-\tfrac{2}{1+\alpha}\hat\rho_{k+1}\rho_{k+1})P_{k-1}(x)
			\\
			&
				-
				\left\{
					\tfrac{\alpha}{1+\alpha}
					-
					\tfrac{2}{1+\alpha}
					\hat\rho_{k+1}
					\rho_{k+1}
					\left[
						\Lambda_{k,\delta_{k+1}}
						-
						(1-\gamma_k\ell_k)
					\right]
				\right\}
				\DD_{\kernel}(x^k,x^{k-1}),
			\end{align*}
			where \(\U^{\alpha}_k(x)\) corresponds to \(\widehat{\U}_k(x)\) as in \cref{thm:eq} with \(\vartheta_k=\tfrac{2}{1+\alpha}\hat\rho_k\).
			By imposing negativity of the coefficients of \(P_{k-1}(x)\) and \(\DD_{\kernel}(x^k,x^{k-1})\) as done in the previous subsection, the following analogue of \cref{thm:BY:leq} for \refFNEa{} is derived.

			\begin{lemma}\label{thm:FNEa:leq}%
				Additionally to \cref{ass:basic}, suppose that \(\kernel\) has symmetry coefficient \(\alpha>0\) and consider the iterates generated by Bregman proximal gradient iterations \eqref{eq:BPG} with \(\gamma_k\) and \(\hat\rho_k\) selected according to \refFNEa.
				Then, denoting
				\begin{equation}\label{eq:UkFNEa}
					\U^{\alpha}_k(x)
				\coloneqq
					\D_{\kernel}(x,x^k)
					+
					\D_{\kernel}(x^k,x^{k-1})
					+
					\gamma_k\left(1+\tfrac{2}{1+\alpha}\hat\rho_k\right)P_{k-1}(x),
				\end{equation}
				one has that
				\begin{equation}\label{eq:UkFNEa:leq}
					\U^{\alpha}_{k+1}(x)
				\leq
					\U^{\alpha}_k(x)
					-
					\gamma_k\Bigl(
						\smash{
							\overbrace{
								1+\tfrac{2}{1+\alpha}\hat\rho_k-\tfrac{2}{1+\alpha}\hat\rho_{k+1}\rho_{k+1}
							}^{
								\geq
								\tfrac{2}{1+\alpha}(\hat\rho_k-\rho_k)
								\geq
								0
							}
						}
					\Bigr)P_{k-1}(x)
				\end{equation}
				holds for any \(x\in\dom\varphi\cap\dom\kernel\) and \(k\in\N\).
			\end{lemma}

	%% ██ 4. Convergence analysis ████████████████████████████████████████████████████████████████████████████████████████
	\section{Convergence analysis}\label{sec:cvg}%
		This section is devoted to proving \cref{thm:main} in its entirety.
		We first provide some technical lemmas that will be invoked in the proofs.
		The first one is a direct consequence of \cref{thm:BY:leq,thm:FNEa:leq}, and its statement closely patterns similar results in the Euclidean setting.
		The simple proof is given in the appendix.

		\begin{lemma}\label{thm:descent}%
			Suppose that \cref{ass:basic} holds, and consider the iterates generated by Bregman proximal gradient iterations \eqref{eq:BPG} with \(\gamma_k\) and \(\hat\rho_k\) selected according to \refBY.
			Then, with \(\U_k\) as in \eqref{eq:UkBY},
			the following hold for any \(x\in\dom\kernel\) satisfying \(\varphi(x)\leq\inf_{k\in\N}\varphi(x^k)\):
			\begin{enumerate}
			\item \label{thm:SD}%
				\(\bigl(\U_k(x)\bigr)_{k\in\N}\) decreases and converges to a finite value.
			\item \label{thm:Pkmin}%
				\(P_K^{\rm min}(x)\leq\frac{\U_0(x)}{\sum_{k=1}^{K+1}\gamma_k}\) for every \(K\geq 1\), where \(P_K^{\rm min}(x)\coloneqq\min_{k\leq K}P_k(x)\).
			\end{enumerate}
			When \(\kernel\) has symmetry coefficient \(\alpha>0\), all remains true for the updates of \refFNEa, with \(\U_k\gets\U^{\alpha}_k\) as in \eqref{eq:UkFNEa}.
		\end{lemma}
		\begin{proof}
			See \cref{proof:thm:descent}.
		\end{proof}

		%% ░░░░ 4.1 Proof outline ░░░░░░░░░░░░░░░░░░░░░░░░░░░░░░░░░░░░░░░░░░░░░░░░░░░░░░░░░░░░░░░░░░░░░░░░░░░░░░
		\subsection{Proof outline}\label{sec:outline}%
			The existence of \(x\in\dom\kernel\) such that \(\varphi(x)\leq\varphi(x^k)\) for all \(k\) is fundamental for the validity of \cref{thm:descent}.
			When \(\dom\kernel=\R^n\), thus in the Euclidean case in particular, this is granted.
			More generally, the possibility of solutions existing merely on the boundary of \(\dom\kernel\) renders this statement inapplicable.
			Nevertheless, this result will prove fundamental in demonstrating \cref{thm:main} in its full generality, and particularly the general claim in \eqref{eq:main:inf}.

			Our proof strategy closely follows the ideas of the seminal work \cite{bauschke2017descent}, where it was shown that
			\[
				\inf_k \varphi(x^k)\le \varphi(u)
			\qquad
				\forall u\in\dom\kernel,
			\]
			and hence in particular that \(\inf_k\varphi(x^k)=\inf_C\varphi\).
			A key departure, however, is that the proof in \cite{bauschke2017descent} applies uniformly regardless of whether or not minimizers exist in \(C\).
			In contrast, our approach necessarily separates the analysis into two distinct regimes.
			The first is based on the existence of an interior point whose objective value lower bounds all iterates; this condition is subsequently shown to characterize the presence of minimizers in the interior.
			The second regime is simply its complement, which precisely corresponds to the case where \(\inf_k \varphi(x^k)=\inf_C\varphi\), with infimum on the right-hand side not attained.
			A conceptual roadmap of our proof strategy can be summarized as follows:
			\begin{itemize}[label={}, leftmargin=0pt]
			\item
				\emph{Special case:} an \(x\in C\) exists such that \(\varphi(x)\leq\inf_{k\in\N}\varphi(x^k)\) as in \cref{thm:descent}.
				\begin{itemize}[widestL]
				\item
					Then \(P_k(x)\ge0\), and hence
					\(
						\D_\kernel(x,x^k)\le\U_k(x)
					\);
					by \cref{thm:SD}, this implies boundedness of \((\D_\kernel(x,x^k))_{k\in\N}\)
					(in fact, all this even if \(x\notin C\)).
				\item
					Since \(x\in C\), \cref{fact:D} implies that \((x^k)_{k\in\N}\) is bounded and bounded away from the boundary of \(C\).
				\item
					The following \cref{thm:Lamkto1} is then applicable and yields \(\gamma_k\not\to0\), from which \cref{thm:Pkmin} implies
					\(
						\inf_k\varphi(x^k)\leq\varphi(x)
					\).
				\item
					By its arbitrariness (and since \(\varphi\) is lsc), any such \(x\) must be optimal.
				\end{itemize}

			\item
				\emph{Complementary case:} no such \(x\) exists.
				\begin{itemize}[widestL]
				\item
					This means that
					\(
						\inf_k\varphi(x^k)=\inf_C\varphi
					\),
					which by convexity further coincides with
					\(
						\inf_{\overline C}\varphi
					\).
				\end{itemize}
			\end{itemize}
			In both cases, standard arguments are then used to establish additional convergence properties.
			Note that the special case turns out to be equivalent to the existence of minimizers in \(C\), and in turn the complementary one corresponds to the case in which solutions exist merely on the boundary.

		%% ░░░░ 4.2 Auxiliary lemmas ░░░░░░░░░░░░░░░░░░░░░░░░░░░░░░░░░░░░░░░░░░░░░░░░░░░░░░░░░░░░░░░░░░░░░░░░░░░░░░
		\subsection{Auxiliary lemmas}%
			A key to address the ``special case'' is the following lemma, which applies when the iterates \(x^k\) stay bounded away from the boundary of \(\dom\kernel\).
			Roughly speaking, under this assumption it ensures that whenever stepsizes drop below a certain threshold, the updates in both \eqref{eq:BY:rhok*} and \eqref{eq:FNEa:rhok*} reduce to \(\gamma_{k+1}=\hat\rho_{k+1}\gamma_k\), see \cref{rem:1/0}, and thus increase.
			This behavior is at the heart of the careful choice of parameters \(\vartheta_k=\hat\rho_k\) for \refBY{} and \(\vartheta_k=\frac{1+\alpha}{2\alpha}\hat\rho_k\) for \refFNEa{}.

			\begin{lemma}\label{thm:Lamkto1}%
				Suppose that \cref{ass:basic} holds, and consider a sequence \((x^k)_{k\in\N}\) generated by Bregman proximal gradient iterations \eqref{eq:BPG} with stepsizes \(\gamma_k>0\).
				Suppose that \((x^k)_{k\in\N}\) is contained in a compact set \(\K\subset C\), and consider the ratio
				\[
					\Lambda_{k,\delta_{k+1}}
				=
					\frac{
						2\D_{\kernel^*}\bigl(\nabla\kernel(x^k)+\delta_{k+1}\bigl[H_k(x^k)-H_k(x^{k-1})\bigr],\nabla\kernel(x^k)\bigr)
					}{
						\delta_{k+1}^2
						\DD_{\kernel}(x^k,x^{k-1})
					}
				\]
				as in \eqref{eq:Lamk}, where \((\delta_k)_{k\in\N}\subset\R_{++}\) is a bounded sequence.
				If \(\gamma_k\to0\), then \(\Lambda_{k,\delta_{k+1}}\to1\).
			\end{lemma}
			\begin{proof}
				For notational conciseness, let us denote \(u^k\coloneqq\nabla\kernel(x^k)-\nabla\kernel(x^{k-1})\) and \(v^k\coloneqq\nabla f(x^k)-\nabla f(x^{k-1})\), so that \(H_k(x^k)-H_k(x^{k-1})=u^k-\gamma_k v^k\).
				For any \(k\in\N\) there exists
				\(\xi^k\) on the line segment between \(\nabla\kernel(x^k)\) and \(\nabla\kernel(x^k)+\delta_{k+1}(H_k(x^k)-H_k(x^{k-1}))\)
				and
				\(\eta^k\) on the line segment between \(\nabla\kernel(x^k)\) and \(\nabla\kernel(x^{k-1})\)
				such that
				\begin{align}
				\label{eq:Lamknum}
					\Lambda_{k,\delta_{k+1}}
				={} &
					\frac{
						2
						\D_{\kernel^*}\bigl(\nabla\kernel(x^k)+\delta_{k+1}\bigl[H_k(x^k)-H_k(x^{k-1})\bigr],\nabla\kernel(x^k)\bigr)
					}{
						\delta_{k+1}^2
						\DD_{\kernel}(x^k,x^{k-1})
					}
				\\
				\nonumber
				={} &
					\frac{
						\innprod{
							\nabla^2\kernel^*(\xi_k)
							\bigl[H_k(x^k)-H_k(x^{k-1})\bigr]
						}{
							H_k(x^k)-H_k(x^{k-1})
						}
					}{
						\innprod{
							\nabla^2\kernel^*(\eta_k)
							u^k
						}{
							u^k
						}
					}
				\\
				\nonumber
				={} &
					\frac{
						\innprod{
							\nabla^2\kernel^*(\xi_k)
							u^k
						}{
							u^k
						}
					}{
						\innprod{
							\nabla^2\kernel^*(\eta_k)
							u^k
						}{
							u^k
						}
					}
					-
					2\gamma_k
					\frac{
						\innprod{
							\nabla^2\kernel^*(\xi_k)
							u^k
						}{
							v^k
						}
					}{
						\innprod{
							\nabla^2\kernel^*(\eta_k)
							u^k
						}{
							u^k
						}
					}
					+
					\gamma_k^2
					\frac{
						\innprod{
							\nabla^2\kernel^*(\xi_k)
							v^k
						}{
							v^k
						}
					}{
						\innprod{
							\nabla^2\kernel^*(\eta_k)
							u^k
						}{
							u^k
						}
					}.
				\end{align}
				Since \(\kernel\) is twice continuously differentiable with \(\nabla^2\kernel\succ0\) on \(C\), and \(\K\subset C\) is compact and convex, one has that
				\[
				\textstyle
					L_{\kernel,\K}
				\coloneqq
					\sup_{\K}\norm{\nabla^2\kernel}
				=
					\left(\inf_{\nabla\kernel(\K)}\lambda_{\rm min}(\nabla^2\kernel^*)\right)^{-1}
				\]
				is finite (\(L_{\kernel,\K}\) being the Lipschitz modulus of \(\nabla\kernel\) on \(\K\)).
				As such, one has that
				\(
					\innprod{\nabla^2\kernel^*(\eta_k)u^k}{u^k}
				\geq
					L_{\kernel,\K}^{-1}\norm{u^k}^2
				\)
				for all \(k\).
				Moreover, letting \(L_{f,\K}^{\kernel}\) denote a smoothness modulus of \(f\) relative to \(\kernel\) on \(\K\), ensured to exist by \cref{ass:f}, we infer from \cite[Prop. 2.5(ii)]{ahookhosh2021bregman} that
				\(
					\norm{v^k}\leq L_{f,\K}^{\kernel} L_{\kernel,\K}\norm{u^k}
				\)
				holds for all \(k\).
				Therefore,
				\begin{align*}
					\abs{\Lambda_{k,\delta_{k+1}}-1}
				\leq{} &
					\norm{\nabla^2\kernel^*(\xi_k)-\nabla^2\kernel^*(\eta_k)}
					L_{\kernel,\K}
					+
					2\gamma_k
					L_{f,\K}^{\kernel} L_{\kernel,\K}^2
					\norm{\nabla^2\kernel^*(\xi_k)}
				\\
				&
					+
					\gamma_k^2
					L_{\kernel,\K}^3
					(L_{f,\K}^{\kernel})^2
					\norm{\nabla^2\kernel^*(\xi_k)}.
				\end{align*}
				If \(\gamma_k\to0\), then eventually \(\gamma_k<\frac{1}{L_{f,\K}^{\kernel}}\) and standard results ensure that \((x^k)_{k\in\N}\) converges to some point \(x\in\K\).
				In this case, \((\xi^k)_{k\in\N}\) and \((\eta^k)_{k\in\N}\) converge to \(\nabla\kernel(x)\) (the former sequence because \((\delta_k)_{k\in\N}\) is bounded), and by continuity of \(\nabla^2\kernel^*\) on \(\K\) the right-hand side in the above inequality vanishes.
			\end{proof}

			The following is another auxiliary result that considers iterates \eqref{eq:BPG} that remain bounded away from the boundary of \(\dom\kernel\).
			It essentially states that whenever a subsequence \((x^k)_{k\in K}\) converges to a solution, then also the shifted subsequence \((x^{k+1})_{k\in K}\) does provided that the corresponding stepsizes do not diverge.

			\begin{lemma}\label{lem:prox:property:optimal}%
				Additionally to \cref{ass:basic} suppose that \(\argmin_C\varphi\neq\emptyset\).
				Let a sequence \((y^k)_{k\in\N}\) contained in \(C\) and converging to a point \(y^\star\in\argmin_C\varphi\) be fixed, and for every \(k\in\N\) let
				\begin{equation}\label{eq:baryk}
					\bar y^k
				\coloneqq
					\argmin_w\left\{
						g(w)
						+
						\innprod{\nabla f(y^k)}{w-y^k}
						+
						\tfrac{1}{\gamma_{k+1}}\D_{\kernel}(w,y^k)
					\right\}
				\end{equation}
				where \((\gamma_k)_{k\in\N}\) is a bounded sequence of strictly positive stepsizes.
				Then, \(\bar y^k\to y^\star\) and \(\gamma_{k+1}(\varphi(\bar y^k) - \min\varphi)\to0\).
			\end{lemma}
			\begin{proof}
				Consider the ``left'' Bregman proximal operator \(\prox_{g}:\interior\dom\kernel\to\dom\kernel\) defined as
				\[
					\prox_{g}(y)
				\coloneqq
					\argmin_{w\in\R^n}\{g(w)+\D_{\kernel}(w,y)\}
				\]
				(note that \(\prox_{g}(y)\in\interior\dom\kernel\subseteq\dom\kernel\) for any \(y\in\interior\dom\kernel\), owing to \cref{ass:basic}).
				The mapping \(\prox_{\gamma_{k+1}g}\circ\nabla\kernel^*\) is \(\nabla\kernel\)-firmly-nonexpansive, in the sense that
				\begin{equation}\label{eq:BPG:FNE}
					\DD_{\kernel}(\bar y_1,\bar y_2)
				\leq
					\innprod{\eta_1-\eta_2}{\bar y_1-\bar y_2}
				\end{equation}
				holds for any \(\eta_i\in\R^2\) and \(\bar y_i = \prox_{\gamma_{k+1} g}\circ\nabla\kernel^*(\eta_i)\), \(i=1,2\);
				a proof of this fact can be found in \cite[Lem. 4.2]{wang2022bregman} (see also \cite[Thm. 4.9]{wang2025bregman} for its \emph{equivalence} to convexity of \(g\)).
				Note that \eqref{eq:baryk} can equivalently be written as
				\[
					\bar y^k
				=
					\prox_{\gamma_{k+1} g}\circ\nabla\kernel^*(H_{k+1}(y^k)),
				\]
				where we remind that \(H_{k+1}=\nabla\kernel-\gamma_{k+1}\nabla f\).
				Using \(\nabla\kernel\)-firm nonexpansiveness and recalling that \(y^\star = \prox_{\gamma_{k+1} g}\circ\nabla\kernel^*(H_{k+1}(y^\star))\) we have
				\begin{align*}
					\DD_{\kernel}(\bar y^k,y^\star)
				\stackrel{\text{\clap{\eqref{eq:BPG:FNE}}}}{\leq}{} &
					\innprod{H_{k+1}(y^k) - H_{k+1}(y^\star)}{\bar y^k - y^\star}
				\\
				\leq{} &
					\Bigl(
						\norm{\nabla \kernel(y^k) - \nabla \kernel(y^\star)}
						+
						\gamma_{k+1}\norm{\nabla f(y^k) - \nabla f(y^\star)}
					\Bigr)
					\norm{\bar y^k - y^\star},
				\end{align*}
				which vanishes as \(k\to\infty\).
				Then, by essential strict convexity of \(\kernel\), it follows that \(\bar y^k\to y^\star\).
				Moreover, by the subgradient inequality, for every \(k \in \N\) it holds that
				\begin{align*}
					0
				\leq
					\gamma_{k+1}(\varphi(\bar y^k) - \min\varphi)
				={} &
					\gamma_{k+1}(f(\bar y^k) + g(\bar y^k) - \min\varphi)\\
				\leq{} &
					\gamma_{k+1}(f(\bar y^k) - f(y^\star))
				\\
				&
					-\innprod{\nabla \kernel(y^k) - \gamma_{k+1}\nabla f(y^k) - \nabla \kernel(\bar y^k)}{y^\star - \bar y^k}.
				\end{align*}
				The proof then follows from continuity of \(f\) and the fact that the inner product vanishes, since both \(y^k\) and \(\bar y^k\) converge to \(y^\star\).
			\end{proof}

			In the remainder of this section, we delve into the proof of \cref{thm:main}.
			We begin by establishing \cref{thm:main:C}, which addresses the unconstrained-like setting where all three lemmas introduced above are directly applicable.
			As also apparent from the proof outline of \cref{sec:outline}, despite the simplifying assumptions, this part of the proof is the most technically involved.
			Once established, the general result in \eqref{eq:main:inf}, and ultimately \cref{thm:main:zone}, will follow as comparatively simpler corollaries.

		%% ░░░░ 4.3 Proof of Theorem 2.7(i) ░░░░░░░░░░░░░░░░░░░░░░░░░░░░░░░░░░░░░░░░░░░░░░░░░░░░░░░░░░░░░░░░░░░░░░░░░░░░░░
		\subsection{Proof of \texorpdfstring{\cref{thm:main:C}}{Theorem \ref*{thm:main:C}}}\label{proofsec:thm:main:C}%
			This entire subsection is dedicated to proving the following result, which covers \cref{thm:main:C} under more general conditions;
			this higher degree of generality will serve as a fundamental intermediate step in the treatment of the more general setting.

			\begin{theorem}\label{thm:xinC}%
				Additionally to \cref{ass:basic}, suppose that
				\begin{equation}\label{eq:xinC}
					\text{there exists \(x\in C\) such that \(\varphi(x)\leq\inf_{k\in\N}\varphi(x^k)\)}
				\end{equation}
				holds for the iterates \(x^k\) generated by \refBY.
				Then, the sequence \((x^k)_{k\in\N}\) converges to a solution \(x^\star\in\argmin_C\varphi\neq\emptyset\) and \(\varphi(x^k)\to\inf_{\overline C}\varphi\).

				When \(\kernel\) has symmetry coefficient \(\alpha>0\), the same is true for \refFNEa.
			\end{theorem}

			The fact that this result subsumes \cref{thm:main:C} is obvious by observing that the validity of \cref{thm:xinC} implies that any \(x\) complying with \eqref{eq:xinC} must necessarily belong to \(\argmin_{\overline C}\varphi\cap C\).
			The proof of \cref{thm:xinC} will be carried out via intermediate claims.
			In what follows, we consider \(\U_k\) as in \eqref{eq:UkBY} in case of \refBY{}; under the needed symmetry assumption, the same proof applies to \refFNEa{} as well by simply replacing \(\U_k\gets\U^{\alpha}_k\).

			\begin{claim}\label{thm:xinC:boundedseq}
				There exists a compact set \(\K\subset C\) containing all the iterates \(x^k\).
			\end{claim}
			\begin{proof}
				In view of \eqref{eq:xinC}, \cref{thm:SD} implies that
				\[
					\D_{\kernel}(x,x^k)
				\leq
					\U_k(x)
				\leq
					\U_0(x)
				\]
				holds for any \(k\in\N\).
				The assertion then follows from \cref{fact:D:coercive}.
			\end{proof}

			\begin{claim}
				\(\gamma_k\not\to0\).
			\end{claim}
			\begin{proof}
				To arrive to a contradiction, suppose that \(\gamma_k\to0\).
				Then, \cref{thm:Lamkto1} implies that eventually the quantity
				\(
					[\Lambda_{k,\delta}-(1-\gamma_k\ell_k)]_+
				\)
				appearing in both \eqref{eq:BY:rhok*} and \eqref{eq:FNEa:rhok*} vanishes.
				Since \((x^k)_{k\in\N}\) is bounded and bounded away from the boundary of \(C\), regardless of whether or not \(\kernel\) has a symmetry coefficient \(\alpha>0\), it holds that \(\inf_{k\in\N}\alpha_k>0\);
				furthermore, for both \refBY{} and \refFNEa{} note that \(\hat\rho_{k+1}\leq\sqrt{1+\rho_k}\leq\sqrt{1+\hat\rho_k}\), implying that \(\hat\rho_k\leq\frac{1+\sqrt{5}}{2}\) holds for any \(k\).
				As such, in both \eqref{eq:BY:rhok*} and \eqref{eq:FNEa:rhok*} eventually the second element in the minimum is infinite, implying that the update reduces to \(\gamma_{k+1}=\hat\rho_{k+1}\gamma_k\) and is thus divergent, a contradiction.
			\end{proof}

			\begin{claim}
				\(\inf_{k\in\N}\varphi(x^k)=\inf_{\overline C}\varphi\) (in particular, necessarily \(x\in\argmin_C\varphi\) for any \(x\) satisfying \eqref{eq:xinC}).
			\end{claim}
			\begin{proof}
				Having shown that \(\gamma_k\not\to0\), one has that \(\sum_{k\in\N}\gamma_k=\infty\).
				\Cref{thm:Pkmin} then implies that \(\inf_{k\in\N}\varphi(x^k)=\varphi(x)\).
				From the arbitrariness of the point \(x\) as in \eqref{eq:xinC} we conclude that necessarily \(x\in\argmin_C\varphi\subseteq\argmin_{\overline C}\varphi\), with inclusion holding by virtue of \cite[Prop. 11.1(iv)]{bauschke2017convex}.
				Hence \(\inf_{k\in\N}\varphi(x^k)=\inf_{\overline C}\varphi\).
			\end{proof}

			\begin{claim}\label{claim:unique}%
				There exists exactly one optimal limit point \(x^\star\in\argmin_C\varphi\).
			\end{claim}
			\begin{proof}
				The existence of an optimal limit point is guaranteed by the previous claim, since \((x^k)_{k\in\N}\) is bounded (and bounded away from the boundary of \(C\)) and \(\varphi\) is lsc.
				Consider two optimal limit points \(\bar x_1\) and \(\bar x_2\), and two corresponding subsequences \((x^k)_{k\in K_1}\to\bar x_1\) and \((x^k)_{k\in K_2}\to\bar x_2\).
				Notice that
				\[
					\U_k(\bar x_1)-\U_k(\bar x_2)
				=
					\D_{\kernel}(\bar x_1,x^k)-\D_{\kernel}(\bar x_2,x^k)
				\]
				converges to some finite value \(V\), because both \(\U_k(\bar x_1)\) and \(\U_k(\bar x_2)\) are convergent.
				Considering the limits along \(k\in K_1\) and \(k\in K_2\) yields that
				\[
					V
				=
					\D_{\kernel}(\bar x_1,\bar x_1)-\D_{\kernel}(\bar x_2,\bar x_1)
				=
					\D_{\kernel}(\bar x_1,\bar x_2)-\D_{\kernel}(\bar x_2,\bar x_2),
				\]
				hence that \(-\D_{\kernel}(\bar x_2,\bar x_1)=\D_{\kernel}(\bar x_1,\bar x_2)\).
				Thus \(\DD_{\kernel}(\bar x_1,\bar x_2)=0\), implying that \(\bar x_1=\bar x_2\).
			\end{proof}

			In light of the previous claim, the proof of \cref{thm:xinC} is completed once we show that \(x^k\to x^\star\).
			To this end, owing to the fact that \(\D_{\kernel}(x^\star,x^k)\leq\U_k(x^\star)\) it will suffice to show that \(U\coloneqq\lim_{k\to\infty}\U_k(x^\star)\) is zero, (existence and finiteness of \(U\) is ensured by \cref{thm:SD}).
			\begin{itemize}
			\item
				Let us consider a subsequence \((x^k)_{k\in K}\) converging to \(x^\star\).
				If the corresponding subsequence \((\gamma_{k+1})_{k\in K}\) is bounded, then it follows from \cref{lem:prox:property:optimal} that  \(x^{k+1}\to x^\star\) and \(\gamma_{k+1} P_k(x^\star) \to 0\) as well, and from the expression of \(\U_k\) together with the fact that \(\hat\rho_{k+1}\) is bounded and \(\inf_k \alpha_k>0\), it is clear that
				\(\U_{k+1}(x^\star)\to0\) as \(K\ni k\to\infty\).

			\item
				In what follows, let us instead consider the complementary case in which for any subsequence \(x^k\) converging to \(x^\star\) the corresponding subsequence of stepsizes \((\gamma_{k+1})_{k\in K}\) is divergent.
				In this case, note that since \(\gamma_{k+1}P_k(x^\star)\leq\U_{k+1}(x^\star)\leq\U_0(x^\star)\), we have
				\begin{align}
				\nonumber
					(x^k)_{k\in K}\to x^\star
				\Leftrightarrow{} &
					(\gamma_{k+1})_{k\in K}\to\infty
				\Leftrightarrow
					\bigl(P_k(x^\star)\bigr)_{k\in K}\to0
				\\
				\nonumber
				\Rightarrow{} &
					(\gamma_k)_{k\in K}\to\infty
				\\
				\Leftrightarrow{} &
					(x^{k-1})_{k\in K}\to x^\star
				\Leftrightarrow
					\bigl(P_{k-1}(x^\star)\bigr)_{k\in K}\to0,
				\label{eq:gamk:infty"gamk-1}
				\end{align}
				where the right implication follows from the fact that \(\gamma_{k+1}\leq\hat\rho_{k+1}\gamma_k\leq\rho_{\rm max}\gamma_k\), for some \(\rho_{\rm max}\leq\frac{1+\sqrt{5}}{2}\).
				Since an optimal limit exists, it follows that \(\sup_{k\in\N}\gamma_k=\infty\).

				In what follows, we construct a specific subsequence \(K\coloneqq\{k_0,k_1,\dots\}\) along which \((\gamma_k)_{k\in K}\) diverges.
				In doing so, we expand upon the arguments in the proof of \cite[Thm. 2.3]{latafat2024adaptive} to account for the complications of the non-Euclidean setting investigated here.
				For \(i\geq0\) let
				\begin{equation}\label{eq:ki+}
					k_{i+1}=\min\{k\geq k_i \mid \gamma_k \geq \rho_{\rm max}\gamma_{k_i}\}.
				\end{equation}
				Then, \(\bigl(\gamma_{k_i}\bigr)_{i\in\N}\to\infty\), which by \eqref{eq:gamk:infty"gamk-1} implies that
				\begin{equation}\label{eq:xkitox*}
					\lim_{i\to\infty}x^{k_i-s}
				=
					x^\star
				\quad
					\forall s=1,2,3
				\end{equation}
				(in fact, for any \(s\in\N\)).
				In light of \cref{thm:SD} we have that
				\begin{equation}\label{eq:U:gamkPk-1}
					U= \lim_{k \to \infty} \U_k= \lim_{i \to \infty} \U_{k_i-1}(x^\star) = \lim_{i \to \infty} \gamma_{k_i-1}(1+ \vartheta_{k_i-1})P_{k_i-2}(x^\star).
				\end{equation}

				We proceed to show that \(\gamma_{k-1} P_{k-2}(x^\star)\) converges to zero along the same subsequence.
				\begin{subequations}
					For every \(i\in\N\), note that
					\begin{gather}
						\rho_{k_i}>1
						~\text{ and }~
						k_{i-1}\leq k_i-1
					\shortintertext{by minimality of \(k_i\) and the fact that \(\rho_{\rm max}>1\), hence that}
					\label{eq:k_i-1}
						\text{either }~
						\rho_{k_i-1}>1
						~\text{ or }~
						\bigl(
							k_i-1\notin K
							\text{ and thus }
							k_{i-1}\leq k_i-2
						\bigr).
					\shortintertext{These combined imply also that}
					\label{eq:k-1}
						\rho_{k_i-1}\geq\rho_{\rm max}^{-1}.
					\end{gather}
					Indeed, if not, then \(\rho_{k_i-1}<\rho_{\rm max}^{-1}<1\), implying by \eqref{eq:k_i-1} that \(k_{i-1}\leq k_i-2\);
					this would lead to the contradiction
					\[
						\rho_{\rm max}\gamma_{k_{i-1}}
					\stackrel{\text{\clap{\eqref{eq:ki+}}}}{\leq}
						\gamma_{k_i}
					\leq
						\rho_{\rm max}\gamma_{k_i-1}
					=
						\rho_{\rm max}\rho_{k_i-1}\gamma_{k_i-2}
					\stackrel{\mathclap{\bot}}{<}
						\gamma_{k_i-2}
					\leq
						\rho_{\rm max}\gamma_{k_{i-1}},
					\]
					where ``\(\bot\)'' marks where the contradictory inequality is used, and the last inequality follows both in case \(k_i-2=k_{i-1}\) (since \(\rho_{\rm max}>1\)) or when \(k_{i-1}<k_i-2<k_i\) (from minimality in the definition of \(k_i\)).
				\end{subequations}
				Therefore, \eqref{eq:k-1} holds true, implying in particular that the sequence
				\begin{equation}\label{eq:Lambda:k-2:bounded}
					\bigl(\Lambda_{k_i-2, 2\hat\rho_{k_i-1}}\bigr)_{i\in\N}
				\quad
					\text{is bounded,}
				\end{equation}
				for otherwise the second term in either \eqref{eq:BY:rhok*} or \eqref{eq:FNEa:rhok*} would vanish, owing to the fact that \(\hat\rho_k\) is lower bounded by \(\sqrt{1/2}\) and similarly \(\alpha_k\) is bounded away from zero.

				Let \(v^k\coloneqq H_k(x^{k-1})- H_k(x^k)\in\gamma_k\partial\varphi(x^k)\), and observe that
				\begin{align*}
				&
					\gamma_k P_{k-1}(x^\star)
				\\
				\leq{} &
					\rho_k\innprod{x^\star - x^{k-1}}{-v^{k-1}}
				\\
				\leq{} &
					\tfrac{\rho_k}{2\hat\rho_k}\D_{\kernel}(x^\star, x^{k-1})
					+
					\tfrac{\rho_k}{2\hat\rho_k}
					\D_{\kernel^*}\bigl(\nabla\kernel(x^{k-1})-2\hat\rho_kv^{k-1}\bigr), \nabla \kernel(x^{k-1})\bigr)
				\\
				={} &
					\tfrac{\rho_k}{2\hat\rho_k}\D_{\kernel}(x^\star, x^{k-1})
					+
					2\rho_k \hat\rho_k^2\Lambda_{k-1, 2\hat\rho_k} \DD_{\kernel}(x^{k-1},x^{k-2})
				\\
				\leq{} &
					\tfrac{1}{2}\D_{\kernel}(x^\star, x^{k-1})
					+
					2\rho_{\rm max}^3 \Lambda_{k-1, 2\hat\rho_k} \DD_{\kernel}(x^{k-1},x^{k-2})
				\end{align*}
				holds for any \(k\), where we used the Bregman--Young inequality \eqref{eq:BY} with parameter \(\delta= \hat\rho_k\) in the second inequality, and \(\rho_k\leq \hat\rho_k\leq \rho_{\rm max}\) in the last one.
				Recall that \(\bigl(x^{k_i-3}\bigr)_{i\in\N} \to x^\star\) by \eqref{eq:xkitox*}.
				Thus,
				\[
					\gamma_{k_i-1} P_{k_i-2}(x^\star)
				\leq
					\tfrac{1}{2}
					\underbrace{
						\D_{\kernel}(x^\star, x^{k_i-2})
						\vphantom{\Lambda_{k_i-2, 2\rho_{k_i-1}}}
					}_{\to 0 \text{ by \eqref{eq:xkitox*}}}
					+
					2\rho_{\rm max}^3
					\underbrace{
						\Lambda_{k_i-2, 2\rho_{k_i-1}}
					}_{\text{\rm bounded}}
					\,
					\underbrace{
						\DD_{\kernel}(x^{k_i-2},x^{k_i-3})
						\vphantom{\Lambda_{k_i-2, 2\rho_{k_i-1}}}
					}_{\to 0 \text{ by \eqref{eq:xkitox*}}}.
				\]
				Using this in \eqref{eq:U:gamkPk-1} and noting that \(\theta_{k_i-1}\) being a multiple of \(\hat{\rho}_{k_i-1}\) in both \refBY{} and \refFNEa{} is bounded, completes the proof.
			\end{itemize}

		%% ░░░░ 4.4 Proof of Equation (2.8) ░░░░░░░░░░░░░░░░░░░░░░░░░░░░░░░░░░░░░░░░░░░░░░░░░░░░░░░░░░░░░░░░░░░░░░░░░░░░░░░░
		\subsection{Proof of \texorpdfstring{\cref{eq:main:inf}}{Equation \ref*{eq:main:inf}}}
			In this subsection we prove that the iterates generated by either \refBY{} or \refFNEa{} in the generality of \cref{ass:basic} (in addition to \(\kernel\) having a symmetry coefficient \(\alpha>0\) in the latter case) are such that \(\inf_{k\in\N}\varphi(x^k)=\inf_{\overline C}\varphi\).

			To see this, contrary to the claim suppose that \(\inf_{k\in\N}\varphi(x^k)>\inf_{\overline C}\varphi\).
			Then, since \(\inf_{\overline C}\varphi=\inf_C\varphi\) by \cite[Prop. 11.1(iv)]{bauschke2017convex}, there exists \(x\in C\) such that \(\varphi(x^k)\geq\varphi(x)\) holds for all \(k\in\N\).
			Invoking \cref{thm:xinC} yields a contradiction.

		%% ░░░░ 4.5 Proof of Theorem 2.7(ii) ░░░░░░░░░░░░░░░░░░░░░░░░░░░░░░░░░░░░░░░░░░░░░░░░░░░░░░░░░░░░░░░░░░░░░░░░░░░░░
		\subsection{Proof of \texorpdfstring{\cref{thm:main:zone}}{Theorem \ref*{thm:main:zone}}}
			In this subsection we prove that whenever \cref{ass:basic,ass:zone} are satisfied, the iterates generated by \refBY{} are bounded and admit exactly one optimal limit point.\footnote{%
				\label{foot:noFNEa}%
				As commented after \cref{thm:main}, there is no loss of generality in considering only \refBY{} and disregard the claim for \refFNEa.
			}
			We can without loss of generality assume that \(\argmin_{C}\varphi=\emptyset\), for otherwise a stronger result is already covered by \cref{thm:main:C}.

			Under \cref{ass:zone}, \(\dom\kernel=\overline C\), and therefore in \cref{thm:SD} one can take any \(x\in\argmin_{\overline C}\varphi\).
			In particular, \(\D_{\kernel}(x,x^k)\leq\U_k(x)\leq\U_0(x)\) holds for any \(k\).
			Since \(\D_{\kernel}(x,{}\cdot{})\) is level bounded by \cref{ass:xkbounded}, boundedness of \((x^k)_{k\in\N}\) follows.
			Moreover, we know from \eqref{eq:main:inf} that \(\varphi(x^k)\to\inf_{\overline C}\varphi\) along a subsequence, and therefore an optimal limit point \(x^\star\) exists.
			To assess its uniqueness, we can argue similarly to the proof of \cref{claim:unique}, with the minor catch that now such limit points are on the boundary of \(C\).
			Considering two optimal limit points \(\bar x_1\) and \(\bar x_2\), and two corresponding subsequences \((x^k)_{k\in K_1}\to\bar x_1\) and \((x^k)_{k\in K_2}\to\bar x_2\), we still have that
			\[
				\U_k(\bar x_1)-\U_k(\bar x_2)
			=
				\D_{\kernel}(\bar x_1,x^k)-\D_{\kernel}(\bar x_2,x^k)
			\]
			converges to some finite value \(V\).
			Again by considering the limit along the two subsequences \(k\in K_1\) and \(k\in K_2\), \cref{ass:Dhto0} yields that
			\[
				-\lim_{k\in K_2}\D_{\kernel}(\bar x_1,x^k)
			=
				\lim_{k\in K_1}\D_{\kernel}(\bar x_2,x^k),
			\]
			hence that
			\[
				\lim_{k\in K_2}\D_{\kernel}(\bar x_1,x^k)
			=
				0
			=
				\lim_{k\in K_1}\D_{\kernel}(\bar x_2,x^k),
			\]
			which by \cref{fact:D:solodov} implies that \(\bar x_1=\bar x_2\).

		%% ░░░░ 4.6 Bregman proximal gradient with backtracking ░░░░░░░░░░░░░░░░░░░░░░░░░░░░░░░░░░░░░░░░░░░░░░░░░░░░░░░░░░░░░░░░░░░░░░░░░░░░░
		\subsection{Bregman proximal gradient with backtracking}\label{sec:BPGls}%
			The proof outline of \cref{sec:outline} extends, in a simplified form, to establish convergence guarantees akin to those of \cref{thm:main} for BPG iterates with stepsizes selected according to a backtracking strategy such as the following:
			fix \(\beta,c\in(0,1)\), and at each iteration
			\par\noindent\phantomsection \label{alg:BPGls}%
			\parbox[t]{0.15\linewidth}{%
				{\bf\BPGls}

				\footnotesize
			}
			\hfill
			\fbox{\parbox[t]{0.82\linewidth}{%
				choose \(\gamma_{k+1}\geq\gamma_k\), and reduce \(\gamma_{k+1}\gets\beta\gamma_k\) until
				\begin{equation}\label{eq:BPGls}
					\D_f(x^{k+1},x^k)
				\leq
					\tfrac{c}{\gamma_{k+1}}\D_\kernel(x^{k+1},x^k)
				\end{equation}
				holds, where \(x^{k+1}\) is as in \eqref{eq:BPG}.
			}}%

			\medskip
			This linesearch variant will be used as a benchmark in the numerical experiments of the following section.
			For this reason, we provide a brief convergence analysis showing that it is well defined and applicable in all settings where the proposed \refBY{} method can be employed.
			Alternative linesearch strategies are available in the literature, such as the general framework of \cite{rebegoldi2018bregman} which enforces descent directly on the objective, or \cite{takahashi2025approximate,fujiki2026approximate} which apply to a different algorithmic update.
			We retain the present scheme due to its conceptual simplicity and its clear emphasis on the role of local relative smoothness of \(f\), independently of the nonsmooth component of the objective.
			Also in this case, the slight departure from the proof strategy of \cite{bauschke2017descent} outlined in \cref{sec:outline} allows us to work under our mild notion of local relative smoothness (see \cref{rem:local}), without requiring a priori boundedness or separation from the boundary.
			In particular, our analysis \emph{does not} necessarily rely on the existence of a relative smoothness modulus on some set containing the iterates.

			\begin{theorem}[convergence of \refBPGls]\label{thm:BPGls}%
				Let \cref{ass:basic} hold (in fact, without \(\nabla^2\kernel\) necessarily existing or being positive definite), and consider the iterates generated by \refBPGls.
				Then, one always has that \((\varphi(x^k))_{k\in\N}\) is decreasing, with
				\begin{equation}\label{eq:BPGls:lim}
					\smash{
						\lim_{k\to\infty}\varphi(x^k)=\inf_{\overline{C}}\varphi.
					}
				\end{equation}
				Moreover,
				\begin{enumerate}
				\item \label{thm:BPGls:descent}%
					For any \(x\in\dom\varphi\cap\dom\kernel\) and \(k\in\N\) it holds that
					\[
						\D_\kernel(x,x^{k+1})
					\leq
						\D_\kernel(x,x^k)
						-
						(1-c)\D_\kernel(x^{k+1},x^k)
						-
						\gamma_{k+1}\bigl[\varphi(x^{k+1})-\varphi(x)\bigr].
					\]
				\item \label{thm:BPGls:zone}%
					If \cref{ass:zone} holds, then \((x^k)_{k\in\N}\) converges to a solution.
				\item \label{thm:BPGls:C}%
					If \(C\cap\argmin_{\overline{C}}\varphi\neq \emptyset\) (equivalently, if \(C\cap\argmin\varphi\neq\emptyset\)), then \((x^k)_{k\in\N}\) converges to a solution in \(C\), and \(\inf_{k\in\N}\gamma_k>0\).
				\end{enumerate}
			\end{theorem}%

			In place of \cref{thm:eq}, the proof relies on the following simpler identity for consecutive iterates, independently of the stepsize selection criterion.
			It is essentially the identity underlying the inequality in \cite[Lem. 5]{bauschke2017descent}, prior to the introduction of global relative smoothness bounds.

			\begin{lemma}\label{thm:BPGls:eq}%
				Suppose that \cref{ass:basic} holds (in fact, without \(\nabla^2\kernel\) necessarily existing or being positive definite), and starting from \(x^0\in C\) consider Bregman proximal gradient iterations \eqref{eq:BPG} with stepsizes \(\gamma_k>0\).
				Then, for any \(x\in\dom\varphi\cap\dom\kernel\) and \(k\in\N\) it holds that
				\begin{align*}
				&
					\D_{\kernel}(x,x^{k+1})
					+
					\gamma_{k+1}P_{k+1}(x)
					+
					\D_{\kernel}(x^{k+1},x^k)
				\\
				={} &
					\D_{\kernel}(x,x^k)
					+
					\gamma_{k+1}\D_f(x^{k+1},x^k)
					-
					\gamma_{k+1}
					\Bigl\{
						\D_f(x,x^k)
						+
						\tD_g(x,x^{k+1})
					\Bigr\},
				\end{align*}
				where \(P_{k+1}(x)\) is as in \eqref{eq:Pk}.
			\end{lemma}
			\begin{proof}
				From the three-point identity of \cref{thm:3p} we have
				\begin{align*}
				&
					\tfrac{1}{\gamma_{k+1}}
					\bigl[
						\D_{\kernel}(x,x^{k+1})
						+
						\D_{\kernel}(x^{k+1},x^k)
						-
						\D_{\kernel}(x,x^k)
					\bigr]
				\\
				={} &
					\tfrac{1}{\gamma_{k+1}}
					\innprod{x-x^{k+1}}{\nabla\kernel(x^k)-\nabla\kernel(x^{k+1})}
				\\
					\text{\small from \eqref{eq:subgrad} }
				={} &
					\innprod{x-x^{k+1}}{\tilde{\nabla} g(x^{k+1})+\nabla f(x^k)}
				\\
				={} &
					g(x)-g(x^{k+1})-\tD_g(x,x^{k+1})
					+
					f(x)-f(x^k)-\D_f(x,x^k)
				\\
				&
					+
					\innprod{x^k-x^{k+1}}{\nabla f(x^k)}.
				\end{align*}
				Observing that
				\(
					\innprod{x^k-x^{k+1}}{\nabla f(x^k)}
				=
					f(x^k)-f(x^{k+1})+\D_f(x^{k+1},x^k)
				\),
				the claim readily follows.
			\end{proof}

			\begin{proof}[Proof of \cref{thm:BPGls}]
				Since any iterate \(x^k\) lies in \(C\) by \cref{ass:g}, local smoothness of \(f\) relative to \(\kernel\) as in \cref{ass:f} yields that \eqref{eq:BPGls} is satisfied in finitely many backtrackings.
				In particular, \(\gamma_{k+1}\) is well defined and strictly positive for any \(k\in\N\).

				The linesearch inequality \eqref{eq:BPGls} combined with \cref{thm:BPGls:eq} yields assertion \ref{thm:BPGls:descent}.
				Additionally, noting that \eqref{eq:BPG} can equivalently be cast as
				\[
					x^{k+1}
				=
					\argmin_{w\in\R^n}\left\{
						\varphi(w)-\D_f(w,x^k)+\tfrac{1}{\gamma_{k+1}}\D_{\kernel}(w,x^k)
					\right\},
				\]
				it also implies that
				\[
					\varphi(x^{k+1})
				\leq
					\varphi(x^k)
					-
					\tfrac{1-c}{\gamma_{k+1}}
					\D_{\kernel}(x^{k+1},x^k)
				\quad
					\forall k\in\N,
				\]
				whence the claimed monotonic decrease of the cost follows.

				As in the proof outline of \cref{sec:outline}, we now consider two separate cases.
				\begin{itemize}[label={\small\textbullet}]%
				\item
					We consider first the case in which an \(x\in C\) exists such that \(\varphi(x^k)\geq\varphi(x)\) for any \(k\in\N\); this in particular covers the assumption of assertion \ref{thm:BPGls:C}.
					Then, the inequality in assertion \ref{thm:BPGls:descent} can be telescoped, implying that both
					\(
						\gamma_{k+1}P_{k+1}(x)
					\)
					and
					\(
						\D_{\kernel}(x^{k+1},x^k)
					\)
					vanish as \(k\to\infty\), and that \((\D_{\kernel}(x,x^k))_{k\in\N}\) is decreasing, and in particular bounded.
					By \cref{fact:D:coercive,fact:D:boundary} there exists a compact set \(\K\subset C\) that contains all the iterates, hence by \cref{ass:f} a constant \(L_{f,\K}^\kernel>0\) such that \(L_{f,\K}^\kernel\kernel-f\) is convex on \(\K\).
					In particular, \(\D_f\leq L_{f,\K}^\kernel\D_\kernel\) on \(\K\), and consequently any \(\gamma_{k+1}\leq\frac{c}{L_{f,\K}^\kernel}\) satisfies the linesearch condition \eqref{eq:BPGls}.
					As a result, \(\gamma_k\geq\gamma_{\rm min}\coloneqq\min\bigl\{\gamma_0,\beta c/L_{f,\K}^\kernel\bigr\}>0\) holds for any \(k\in\N\).
					It follows that, additionally to \(\gamma_{k+1}P_{k+1}(x)\), \(P_{k+1}(x)\) too vanishes, hence that, by lower semicontinuity, all limit points \(x^\star\) satisfy \(\varphi(x^\star)\leq\varphi(x)\).
					From the arbitrariness of \(x\) we conclude that \(\varphi(x^\star)=\min_C\varphi=\min_{\overline C}\varphi\), with last identity owing to \cite[Prop. 11.1(iv)]{bauschke2017convex}, hence that all limit points are (in \(C\) and) optimal.
					Boundedness of the sequence implies that one limit point exists, be it \(x^\star\).
					Since \(\D_\kernel(x^\star,x^k)\) is decreasing, and \(x^\star\in C\), the entire sequence \((x^k)_{k\in\N}\) converges to \(x^\star\).

				\item
					Consider now the complementary case in which no \(x\in C\) exists such that \(\varphi(x^k)\geq\varphi(x)\) holds for all \(k\in\N\).
					This means that \(\inf_{k\in\N}\varphi(x^k)=\inf_C\varphi=\inf_{\overline C}\varphi\), where again we used \cite[Prop. 11.1(iv)]{bauschke2017convex} for the last identity.
					Since \(\varphi(x^k)\) is decreasing, \eqref{eq:BPGls:lim} follows.
				\end{itemize}
				To conclude, it remains to prove assertion \ref{thm:BPGls:zone} under \cref{ass:zone}.
				In this case, any solution \(x\) belongs to \(\dom\kernel\) and thus qualifies for the inequality in assertion \ref{thm:BPGls:descent}.
				\Cref{ass:xkbounded} yields boundedness of the sequence, which combined with \eqref{eq:BPGls:lim} entails the existence of an optimal limit point, be it \(x^\star\).
				Further let \(K\subseteq\N\) be a subset of indices such that \((x^k)_{k\in K}\) converges to \(x^\star\).
				Since \((\D_\kernel(x^\star,x^k))_{k\in\N}\) is decreasing and vanishes as \(K\ni k\to\infty\) by virtue of \cref{ass:Dhto0}, the entire sequence vanishes and thus \((x^k)_{k\in\N}\to x^\star\) owing to \cref{fact:D:solodov}.
			\end{proof}

\clearpage
	%% ██ 5. Numerical experiments ███████████████████████████████████████████████████████████████████████████████████████
	\section{Numerical experiments}\label{sec:simulations}%
		In this section, we evaluate the performance of the proposed algorithms on a series of standard simulation problems.
		Except for the Euclidean simulations of \cref{sec:simulations:Euclidean} that exploit available Julia code,\footnote{%
			\url{https://github.com/pylat/adaptive-proximal-algorithms}
			\cite[\S4]{latafat2024adaptive}
		}
		all experiments were conducted using MATLAB R2025b.\footnote{%
			Code and experiments are made available on github at \url{https://github.com/OuHongjia/B-adaPG-simulations}.
		}

		In each test problem, we compare only those algorithms that are compatible with the problem's structure and domain (see \cref{table:algorithms}).
		In the convergence plots we report the cost against the number of calls to the (Bregman) proximal gradient oracle; except for \ABPGg{}, in all compared algorithms this coincides with the iteration count.
		For better visualization and comparison across different methods, the cost profiles are normalized as \(\frac{\varphi(x^k) - \min_{\overline C}\varphi}{\varphi(x^0) - \min_{\overline C}\varphi}\).
		The value of \(\min_{\overline C}\varphi\) is retrieved numerically by running Bregman proximal gradient with linesearch \refBPGls{} starting at the best iterate attained by all the tested algorithms.\footnote{%
			\refBPGls{} was selected because it is the only Bregman method among those considered that guarantees a decrease in the cost function at every iteration.
		}

		We also plot the stepsizes in a window of consecutive iterations for all adaptive Bregman methods.
		In test problems where a global relative smoothness constant \(L_f^{\kernel}\) is available, the stepsize plots are normalized by \(1/L_f^{\kernel}\) and the baseline \(\gamma_kL_f^\kernel=1\) is emphasized with a light gray line.
		Unless differently specified, algorithms are terminated when either \(\D_\kernel(x^k,x^{k-1})<10^{-12}\) or \(\norm{\tilde{\nabla}\varphi(x^k)}\leq10^{-9}\), where \(\tilde{\nabla}\varphi(x^k)\in\partial\varphi(x^k)\) is as in \eqref{eq:subgradvarphi} (this latter criterion is only triggered when approaching unconstrained minimizers).

		%% ░░░░ 5.1 Compared algorithms ░░░░░░░░░░░░░░░░░░░░░░░░░░░░░░░░░░░░░░░░░░░░░░░░░░░░░░░░░░░░░░░░░░░░░░░░░░░░░░░░░░
		\subsection{Compared algorithms}
			Our stepsize selection \refBY{} is compared against other Bregman methods and, when applicable, the variant \refFNEa{} and Euclidean strategies.
			A list of all the algorithms is synopsized in \cref{table:algorithms}, together with a schematic summary of the main requirements for each.
			More detailed descriptions are provided in the following subsections.

			\begin{table}[t]
				\centering
				\noindent
				\rowcolors{1}{white}{gray!8}%
				\begin{tabular}{l@{~~} l<{~~} @{}c@{} c@{} c@{} c@{}}
					&
					& \multicolumn{1}{p{1.3cm}}{\small\centering\(L_f^{\kernel}\)-smad}
					& \multicolumn{1}{p{1.3cm}}{\small\centering\(\kernel\) str cvx}
					& \multicolumn{1}{p{1.3cm}}{\small\centering\(\alpha(\kernel)>0\)}
					& \multicolumn{1}{p{1.3cm}}{\small\centering\(C=\R^n\)}
				\\
					\refBY &&&&&
				\\
					\refFNEa &&&& \xmark &
				\\
					\refBPGls &&&&&
				\\
					\BaGRA{} & \cite{tam2023bregman} & & \xmark & &
				\\
					\ABPGg{} & \cite{hanzely2021accelerated} & \xmark & & &
				\\
					\PGls & & & & & \xmark
				\\
					\adaPG{} & \cite{latafat2024adaptive} & & & & \xmark
				\\
					\adaPGq{} & \cite{latafat2024convergence} & & & & \xmark
				\end{tabular}
				\caption[List of algorithms used in the numerical experiments of this section and their standing requirements]{%
					List of algorithms used in the numerical experiments of this section and their standing requirements.
					The term \(L_f^{\kernel}\)-smad, short for \emph{\(L_f^{\kernel}\)-smooth adaptable}, is borrowed from \cite{bolte2018first} to denote global smoothness of \(f\) relative to \(\kernel\) with (known) constant \(L_f^{\kernel}\).
				}%
				\label{table:algorithms}%
			\end{table}

			\subsubsection{Proposed adaptive methods (\refBY{} and \refFNEa{})}\label{sec:simulation:BadaPG}%
				The adaptive stepsize selection \refBY{} is tested on all problems, as the generality of \cref{ass:basic} suffices for its applicability.
				The variant \refFNEa{} is only tested on those instances in which \(\kernel\) has a symmetry coefficient \(\alpha(\kernel)>0\).

				The initial stepsizes are chosen following the strategy proposed in \cite[\S2.1.1]{latafat2024adaptive} for the Euclidean setting, noting that the same advantages extend naturally to the more general Bregman framework considered here. We first generate a trial point \(\tilde{x}\) by performing a single Bregman proximal gradient step from the initial point \(x^0\), using a stepsize \(\gamma_{\mathrm{init}}\) (chosen as \(\gamma_{\rm init} = 1/L_f^{\kernel}\) whenever a global smoothness modulus \(L_f^{\kernel}\) exists).
				Then, using \(x^0\) and \(\tilde{x}\), we compute a local relative smoothness constant \(\ell_0\) via \eqref{eq:lk}, and use its reciprocal \(\gamma_0 = \frac{1}{\ell_0}\) as a refined initial stepsize.
				If \(\gamma_0\) is significantly smaller than \(\gamma_{\rm init}\) (say, \(\gamma_0<0.1\gamma_{\rm init}\)), we reset \(\gamma_0\) to \(\gamma_{\rm init}\) and repeat the initialization procedure until a reasonable stepsize \(\gamma_0\) is obtained. We then proceed to select \(\gamma_{-1}\) small enough such that \(\gamma_0\hat\rho_0 \geq 1/2\ell_0\).

				The overhead of gradient evaluations caused by this selection is fairly accounted for in the plots.
				The same initialization is also chosen for the linesearch methods described next.

			\subsubsection{Linesearch methods (\refBPGls{} and \PGls{})}
				We consider the Bregman proximal gradient method equipped with a linesearch procedure \refBPGls{} as outlined in \cref{sec:BPGls}.
				It is applicable to any problem satisfying \cref{ass:basic} (in fact, even when \(\kernel\) is not twice differentiable); see \cref{thm:BPGls}.
				\PGls{} denotes its Euclidean counterpart, which applies in the unconstrained setting \(C=\R^n\).
				We set the backtracking parameters as \(\beta=\frac{5}{6}\) and \(c=0.95\) in our experiments.

				At each iteration, both methods perform a tentative update and then evaluate whether \(\frac{c}{\gamma_{k+1}}\) is a viable local relative smoothness modulus for \(f\) between consecutive iterates:
				if the condition in \eqref{eq:BPGls} is met, the update is accepted and the next iteration proceeds; otherwise, the stepsize is reduced by a factor \(\beta\) and a new trial is initiated.
				To reduce the number of failed attempts, the trial stepsize is initialized close to the last accepted one.
				In our experiments, we warm-start the stepsize as \(\beta^{-1}=1.2\) times the previously accepted value.
				This modest increase helps avoid overly conservative behavior and significantly improves performance by enabling the stepsize to recover from previously small values.

				Note that each iteration of linesearch-based methods involves additional function evaluations to determine an acceptable stepsize.
				These overheads are \emph{not} reflected in our metrics which only count the number of gradient evaluations; remarkably, even without factoring in the additional function evaluations incurred by linesearch, our method consistently achieves superior performance.

			\subsubsection{Bregman adaptive Golden ratio algorithm (\BaGRA)}
				The adaptive method \BaGRA{} proposed in \cite[Alg. 3]{tam2023bregman} is the Bregman extension of aGraal \cite{malitsky2020golden} that applies to the more general setting of variational inequalities.
				When specialized to the problem \eqref{eq:P} investigated here, it addresses the unconstrained minimization of \(f+g\), provided that a minimizer exists in \(\dom\kernel\), that \(\nabla f\) is Lipschitz on the set \(\dom g\cap\dom\kernel\), assumed bounded \cite[Thm. 4.1]{tam2023bregman}, and that \(\kernel\) is strongly convex.
				At each iteration, the algorithm computes an adaptive stepsize based on a local Lipschitz estimate (in our notation):
				\[
					\gamma_{k+1}
				=
					\min\left\{
						\rho\gamma_k
					,\;
						\frac{\sigma_{\kernel} \nu \rho_k}{4 \gamma_k} \cdot \frac{\norm{x^k - x^{k-1}}^2}{\norm{\nabla f(x^k) - \nabla f(x^{k-1})}^2}
					,\;
						\gamma_{\rm max}
					\right\},
				\]
				where \(\rho_{k+1} = \frac{\gamma_{k+1} \nu}{\gamma_k} \), \(\rho \in [1, \tfrac1{\nu} + \tfrac1{\nu^2}]\) and \( \sigma_{\kernel} \) is the strong convexity parameter of the Bregman kernel \(\kernel\).
				We refer to \cite[Alg. 3]{tam2023bregman} for details of the iterates.
				Following the choices made in \cite{tam2023bregman}, we used \( \nu = 1.5 \), \( \rho = \tfrac1{\nu} + \tfrac1{\nu^2} \), and maximum stepsize \(\gamma_{\rm max} = 10^6\).
				As suggested in \cite{tam2023bregman}, the initial stepsize \(\gamma_0\) is determined by introducing a small random perturbation to the starting point \(x^0\) to obtain a nearby point \(\bar{x}^0\), and then computing the local Lipschitz estimate \(L_0 = \frac{\norm{\nabla f(x^0) - \nabla f(\bar x^0)}}{\norm{x^0 - \bar x^0}}\); the initial stepsize is then set as \(\gamma_0 = \tfrac{1}{L_0}\).

				In all simulations, this method performs consistently worse compared to the other adaptive Bregman algorithms.
				On the one hand, this can be attributed to its broader applicability beyond minimization problems; on the other, we believe the culprit lies in its reliance on a \emph{Lipschitz}-based stepsize update, computed as a ratio of \emph{Euclidean norms}, while the actual updates are carried out in the Bregman geometry.
				In contrast, the proposed \refBY{} and \refFNEa{} schemes leverage \emph{purely Bregman-based estimates}, providing a more faithful description of the problem landscape and iteration updates.

			\subsubsection{Accelerated BPG with gain adaptation (\ABPGg)}
				The \ABPGg{} algorithm is an adaptive variant of the \ABPG{} method, both proposed in \cite{hanzely2021accelerated}.
				These methods exploit the so-called \emph{triangle scaling property} under Bregman geometry to achieve a convergence rate of \(\mathcal{O}(k^{-\gamma})\) via an extrapolation step, where \(\gamma\in(0, 2]\) is known as the \emph{triangle scaling exponent} (TSE).
				This class of methods is applicable when the global relative smoothness constant exists and is known.

				Compared to the \ABPG{} method, \ABPGg{} achieves faster convergence by enforcing an optimal exponent \(\gamma_{\rm in}=2\), enabled by dynamically adjusting a certain ``gain'' coefficient.
				For this reason, we only compare against the latter.
				The dynamic adjustment is validated via a linesearch process.
				Differently from \refBPGls{} and \PGls{} which only involve additional cost evaluations, every failed attempt of the backtracking in \ABPGg{} incurs an extra full Bregman proximal gradient computation, which is accounted for in the cost plots.
				We refer to \cite[Alg. 3]{hanzely2021accelerated} for further details of the iteration process, where the parameters are here set as \(\gamma = 2\), \(\rho=1.1\), and 	\(G_{\rm min}=10^{-3}\).

		%% ░░░░ 5.2 Unconstrained minimization with Hessian norm growing as a polynomial ░░░░░░░░░░░░░░░░░░░░░░░░░░░░░░░░░
		\subsection{Unconstrained minimization with Hessian norm growing as a polynomial}\label{sec:Hessian}%
			As a benchmark to test all algorithms in \cref{table:algorithms}, we consider the problem proposed in \cite[\S2]{lu2018relatively} of minimizing a smooth function whose Hessian grows polynomially in norm.
			The problem is formulated as
			\begin{equation}\label{eq:Ppoly}
			    \minimize_{x\in\R^n}
			    \tfrac{1}{4}\norm{Ax-b}^4_4
			    +
			    \tfrac{1}{2}\norm{Cx-d}^2_2,
			\end{equation}
			where \(A, C\in \R^{m\times n}\) are nonzero matrices, and \(b, d\in \R^m\).
			The cost function is not smooth relative to the Euclidean kernel \(\j\) (i.e., its gradient is not globally Lipschitz differentiable);
			instead, it is smooth relative to
			\begin{equation}\label{eq:kernelPoly}
			    \kernel(x) = \tfrac{1}{4}\norm{x}_2^4 + \tfrac{1}{2}\norm{x}_2^2,
			\end{equation}
			with modulus
			\[
				L_f^{\kernel} = 3\norm{A}^4 + 6\norm{A}^3\norm{b}_2 + 3\norm{A}^2\norm{b}_2^2 + \norm{C}^2,
			\]
			see \cite[p. 339]{lu2018relatively}.
			This kernel \(\kernel\) has a symmetry coefficient \(\alpha(\kernel)=2-\sqrt{3}\) \cite[Tab. 1 and Thm. 5.2]{nilsson2025symmetry}.
			On the one hand, since the minimization is carried over the whole space \(\R^n\) (as opposed to a proper convex subset \(\overline C\)), the problem can be addressed with standard (proximal) gradient iterations with suitably chosen stepsizes.
			On the other hand, the smoothness relative to \(\kernel\) and the absence thereof relative to \(\j\) indicate that employing Bregman algorithms exploiting this tailored kernel should prove beneficial.
			Our simulations confirm this intuition, demonstrating the utility of Bregman algorithms even in the unconstrained setting.

			The matrices \(A, C\) are generated with independent identically distributed entries drawn from the uniform distribution on \([0,1]\), and the corresponding response vectors \(b, d\) are constructed by adding scaled uniform noise to the exact linear outputs.
			Comparisons were conducted across problems of varying sizes using synthetic data.

			\begin{figure}[htb]
			    \centering
			    {%
		\tikzexternalenable
		\pgfkeys{/pgf/images/include external/.code={\includegraphics[width=\linewidth]{#width=\linewidth}}}%
		\tikzsetnextfilename{HessianGrow}% save picture pdf with same name as tex file in \tikzfolder
		\input{./TeX/Tikz/HessianGrow.tex}%
	}
			    {%
		\tikzexternalenable
		\pgfkeys{/pgf/images/include external/.code={\includegraphics[width=\linewidth]{#width=\linewidth}}}%
		\tikzsetnextfilename{HessianGrow_step}% save picture pdf with same name as tex file in \tikzfolder
		\input{./TeX/Tikz/HessianGrow_step.tex}%
	}%
			    \caption[Hessian growing as a polynomial in \(\ell_2\) norm]{%
					Problem \eqref{eq:Ppoly} in \S\ref{sec:Hessian}: Hessian growing as a polynomial in \(\ell_2\) norm.
					Top:
						performance comparisons among all algorithms listed in \cref{table:algorithms} in terms of cost.
					Bottom:
						stepsize variation (normalized by \(L_f^{\kernel}\)) for Bregman methods with adaptive stepsizes in a window of the first 200 iterations.
				}%
			    \label{fig:Hessian}
			\end{figure}

			As evident from \cref{fig:Hessian} (top row), \refBY{} and \refFNEa{} emerge as clear winners in being able to adjust the stepsizes more effectively than with the trial-and-error process of the linesearch.
			The slow convergence of the accelerated algorithm \ABPGg{} owes to the high inner iteration cost for adjusting the parameters, which involves calls to the Bregman proximal gradient oracle (the method is the fastest when measured purely in terms of iteration count).

			The bottom row in \cref{fig:Hessian} illustrates the stepsize behavior of the adaptive Bregman methods.
			The stepsizes produced by \refBY{}, \refFNEa{}, and \refBPGls{} oscillate around comparable averages, namely 6 to 7 orders of magnitude larger than the conservative baseline \(1/L_f^{\kernel}\).
			Among these, \refBY{} and \refFNEa{} exhibit notably higher variability, echoing similar empirical observations for their Euclidean counterparts \cite{latafat2024adaptive,latafat2024convergence,oikonomidis2024adaptive}.
			Such oscillatory stepsize patterns, once regarded as a side effect, have recently received theoretical justification for their efficiency in the unconstrained Euclidean case \cite{grimmer2025composing}.

		%% ░░░░ 5.3 Relative-entropy nonnegative regression ░░░░░░░░░░░░░░░░░░░░░░░░░░░░░░░░░░░░░░░░░░░░░░░░░░░░░░░░░░░░░░
 		\subsection{Relative-entropy nonnegative regression}\label{sec:PPoisson}%
 			We next test the efficacy of the algorithms when a kernel function \(\kernel\) without full domain is employed.
 			The corresponding problems will thus be constrained on (the closure of) \(\dom\kernel\).
 			This simulation is adapted from \cite[\S5.3]{bauschke2017descent} and concerns a nonnegative Poisson linear inverse problem.
 			The problem is formulated as
 			\begin{equation}\label{eq:PPoisson}
 				\minimize_{x\in \R^n_+}
 				\KL(Ax\mid b)
 				+
 				\lambda\norm{x}_1,
 			\end{equation}
 			where $A\in \R^{m\times n}_+$ and $b\in \R^m_{++}$.
 			The KL-divergence is defined as
 			\[
 				\KL(x\mid y)
 				\coloneqq
 				\sum_{i=1}^n\left(x_i\ln\tfrac{x_i}{y_i} - x_i + y_i\right),
 			\]
 			and is precisely the Bregman distance \(\D_{\kernel}(x,y)\) with \(\kernel\) being the \emph{Boltzmann--Shannon entropy}
 			\begin{equation}\label{eq:BSe}
 				\kernel(x)
 			=
 				\sum_{i=1}^{n}
 				\bigl(x_i\ln x_i-x_i\bigr).
 			\end{equation}
 			As such, \(f(x)\coloneqq \KL(Ax\mid b)=\D_{\kernel}(Ax,b)\) is \(L_f^{\kernel}\)-smooth relative to \(\kernel\) with
 			\[
 				L_f^{\kernel}
 			=
 				\max_{1\leq j\leq n}\norm{A_{:j}}_1,
 			\]
 			where \(A_{:j}\) denotes the \(j\)-th column of \(A\).
 			Note that the Boltzmann--Shannon entropy complies with \cref{ass:zone} \cite[Rem. 4]{bauschke2017descent}, and admits a uniform TSE of \(1\) \cite[\S2]{hanzely2021accelerated}.

 			\begin{figure}[t]
 				\centering
 				{%
		\tikzexternalenable
		\pgfkeys{/pgf/images/include external/.code={\includegraphics[width=\linewidth]{#width=\linewidth}}}%
		\tikzsetnextfilename{KL}% save picture pdf with same name as tex file in \tikzfolder
		\input{./TeX/Tikz/KL.tex}%
	}
 				{%
		\tikzexternalenable
		\pgfkeys{/pgf/images/include external/.code={\includegraphics[width=\linewidth]{#width=\linewidth}}}%
		\tikzsetnextfilename{KL_step}% save picture pdf with same name as tex file in \tikzfolder
		\input{./TeX/Tikz/KL_step.tex}%
	}%
 				\caption[KL-divergence nonnegative regression]{%
					Problem \eqref{eq:PPoisson} in \S\ref{sec:PPoisson}:
 					KL-divergence nonnegative regression.
 					Top:
 						convergence in terms of cost.
 					Bottom:
 						stepsize magnitudes in a window of the first 200 iterations (normalized by \(L_f^{\kernel}\)).
 				}%
 				\label{fig:KL}%
 			\end{figure}

 			For each experiment, the matrix \(A\) is sampled with i.i.d. entries from \([0,1]\) and normalized to sum to one.
 			The response vector \(b\) is formed by perturbing the exact output with scaled uniform noise and is strictly positive.
 			We set \(\lambda = 0.001\) and generate synthetic data with varying dimensions to compare the performance of different algorithms.
 			As evident from \cref{fig:KL} (top row), the accelerated algorithms \ABPG{} and \ABPGg{} perform remarkably well in this problem setup, while our adaptive stepsize selection strategies perform slightly better than the linesearch method.
 			The bottom row offers a snapshot of the stepsize magnitude of adaptive methods on the first 200 iterations, showcasing a fluctuating trend consistent with the observation in \cref{fig:Hessian}.
 			This time, the stepsizes oscillate around a value only slightly higher than the baseline \(\frac{1}{L_f^{\kernel}}\).

		%% ░░░░ 5.4 Relative-entropy barrier minimization on the simplex ░░░░░░░░░░░░░░░░░░░░░░░░░░░░░░░░░░░░░░░░░░░░░░░░░░
		\subsection{Relative-entropy barrier minimization on the simplex}\label{sec:KLS}%
			This problem involves the minimization of a generalized volumetric barrier function over the probability simplex:
			\begin{equation}\label{eq:KLS}
				\minimize_{x\in\Delta_n}f(x)\coloneqq\ln\det\bigl(H X^{-1} H^T \bigr),
			\end{equation}
			where \(H\in\R^{m\times n}\) with \(n\geq m+1\), \(X = \diag(x)\), and \(\Delta_n\) is the probability simplex
			\[
			\textstyle
				\Delta_n
			\coloneqq
				\{x\in\R_+^n \mid \sum_{i=1}^nx_i=1\}.
			\]
			\begin{figure}[t]
				\centering
				{%
		\tikzexternalenable
		\pgfkeys{/pgf/images/include external/.code={\includegraphics[width=\linewidth]{#width=\linewidth}}}%
		\tikzsetnextfilename{KLSymplex}% save picture pdf with same name as tex file in \tikzfolder
		\input{./TeX/Tikz/KLSymplex.tex}%
	}
				{%
		\tikzexternalenable
		\pgfkeys{/pgf/images/include external/.code={\includegraphics[width=\linewidth]{#width=\linewidth}}}%
		\tikzsetnextfilename{KLSymplex_step}% save picture pdf with same name as tex file in \tikzfolder
		\input{./TeX/Tikz/KLSymplex_step.tex}%
	}%
				\caption[Relative entropy barrier minimization on the simplex using LIBSVM datasets]{%
					Problem \eqref{eq:KLS} in \S\ref{sec:KLS}:
					Relative entropy barrier minimization on the simplex using LIBSVM datasets.
					Top: convergence in terms of cost.
					Bottom: stepsize magnitudes in a window of the first 200 iterations.
				}%
				\label{fig:KLSDataset}
			\end{figure}%
			As observed in \cite{lu2018relatively}, \(f\) is smooth relative to the \emph{Burg entropy} \(x\mapsto-\sum_{i=1}^n\ln x_i\).
			However, the corresponding Bregman projection onto \(\Delta_n\) is not available in closed form.
			For this reason, we instead adopt the Boltzmann--Shannon entropy \eqref{eq:BSe} as in the previous experiments, whose associated Bregman projection onto \(\Delta_n\) does admit a simple closed-form expression:
			for any \(v\in\R^n\) and \(y\in\R_{++}^n\), the minimizer of
			\[
				\argmin_{x\in \Delta_n}\{\innprod{v}{x} + \D_{\kernel}(x,y)\}
			\]
			is given by
			\begin{equation}
				x_i = \frac{y_i e^{-v_i}}{\sum_j y_j e^{-v_j}}, \quad i=1,\dots,n.
			\end{equation}
			Formally, the problem is cast in form \eqref{eq:P} as
			\[
				\minimize_{x\in\R_n^+}\ln\det\bigl(H X^{-1} H^T \bigr)+\indicator_{\Delta_n}(x),
			\]
			where \(\indicator_{\Delta_n}:\R^n\to\Rinf\) is the \emph{indicator function} of the set \(\Delta_n\), namely such that \(\indicator_{\Delta_n}(x)=0\) for \(x\in\Delta_n\) and \(\indicator_{\Delta_n}(x)=\infty\) otherwise.

			Only algorithms able to cope with lack of relative smoothness of \(f\) and strong convexity and real valuedness of \(\kernel\) from \cref{table:algorithms} are employed.
			We conducted experiments on regression datasets from the LIBSVM repository \cite{chang2011libsvm}, aiming to evaluate its ability to identify the most relevant data points for predicting associated labels;
			specifically, the bodyfat dataset (\(n=252\), \(m=14\)), the mpg dataset (\(n=392\), \(m=7\)), and the housing dataset (\(n=506\), \(m=13\)).
			Also in this case, the proposed adaptive method \refBY{} exhibits superior performance over the linesearch variant.

			We also remark that all solutions in the considered problems exhibit high sparsity (bodyfat: {83.73}\%, mpg: {96.43}\%, housing: {92.29}\%), and thus lie on the boundary of \(\dom\kernel\).
			The bottom plots in \cref{fig:KLSDataset} nonetheless demonstrate that the proposed \refBY{} generates stepsizes that stay bounded away from zero, although not theoretically guaranteed by \cref{thm:main:zone}.
			Whether a rigorous confirmation of this trend can be established is currently under investigation.

		%% ░░░░ 5.5 Entropic mirror descent for underdetermined linear systems ░░░░░░░░░░░░░░░░░░░░░░░░░░░░░░░░░░░░░░░░░░░
		\subsection{Entropic mirror descent for underdetermined linear systems}\label{sec:simulations:linsys}%
			We consider the underdetermined nonnegative linear system studied in \cite{malitsky2026entropic}, cast as the minimization problem
			\begin{equation}\label{eq:Plinsys}
				\minimize_{x\in\R^n_+}
				f(x)
				\coloneqq
				\tfrac{1}{2}\norm{Ax - b}^2,
			\end{equation}
			where \(A\in\R^{m\times n}\) with \(m<n\), and the feasibility set \(\{x\in\R^n_+\mid Ax=b\}\) is assumed nonempty.
			As in \cref{sec:PPoisson}, we employ the Boltzmann--Shannon entropy \eqref{eq:BSe} as the kernel \(\kernel\), with \(g=0\) (so the nonnegativity constraint is encoded entirely through \(\dom\kernel=\R^n_+\)).
			Beyond its practical role as a domain-encoding kernel, the Boltzmann--Shannon entropy induces an \emph{implicit bias} towards \(\ell_1\)-sparse solutions \cite{malitsky2026entropic}.
			Note that the methods \ABPGg{} and \BaGRA{} are inapplicable here, since \(f\) is not globally relative smooth and \(\kernel\) is not strongly convex.

			In this section, we compare against the entropic mirror descent method with Polyak's stepsize proposed in \cite[Eq.~(3.5)]{malitsky2026entropic}, hereafter denoted \MDPolyak{}.
			This method is tailored specifically for this problem, and reads
			\begin{equation}\label{eq:MDPolyak}
				x^{k+1}
			=
				x^k\circ\exp\bigl(-\alpha_k\nabla f(x^k)\bigr),
				\quad
				\alpha_k
			=
				\min\!\left\{
					\tfrac{f(x^k)}{\norm{\nabla f(x^k)}^2_{x^k}}
				,\,
					\tfrac{1.79}{\norm{\nabla f(x^k)}_\infty}
				\right\},
			\end{equation}
			where \(\norm{v}^2_x\coloneqq\innprod{x}{v\circ v}\) is a weighted squared norm and \(\circ\) denotes elementwise (Hadamard) multiplication.
			Notice that the BPG iterate \eqref{eq:BPG} with kernel \eqref{eq:BSe} and \(g=0\) produces exactly the same update; the two methods therefore differ solely in their stepsize rules:
			while \refBY{} derives \(\gamma_{k+1}\) from consecutive gradient evaluations and requires no additional information, \MDPolyak{} leverages knowledge of the optimal value \(f^*=0\) via the Polyak formula.

			Following the setup of \cite{malitsky2026entropic}, we generate a matrix \(A\in\R^{300\times 500}\) with i.i.d.\ entries drawn from the half-normal distribution, and a sparse ground truth \(z\in[0,1]^{500}\) with \(\norm{z}_0\in\{0,30,60\}\) nonzero entries, setting \(b=Az\).
			All algorithms are initialized at \(x^0=10^{-4}\mathbf{1}\), where \(\mathbf{1}\) is the vector of all ones.

			This toy simulation is intended to assess the performance of the proposed \refBY{} against a Bregman-based solver tailored for a specific problem.
			Not surprisingly, as evident from \cref{fig:LinSys} (top) \MDPolyak{} consistently outperforms our \refBY{}.
			Nevertheless, while operating under considerably weaker assumptions and applying to a substantially broader setting, \refBY{} showcases a robust performance and employs stepsizes of comparable magnitutes even in this specialized setting.

			\begin{figure}[htb]
				\centering
				{%
		\tikzexternalenable
		\pgfkeys{/pgf/images/include external/.code={\includegraphics[width=\linewidth]{#width=\linewidth}}}%
		\tikzsetnextfilename{Entropy}% save picture pdf with same name as tex file in \tikzfolder
		\input{./TeX/Tikz/Entropy.tex}%
	}
				{%
		\tikzexternalenable
		\pgfkeys{/pgf/images/include external/.code={\includegraphics[width=\linewidth]{#width=\linewidth}}}%
		\tikzsetnextfilename{Entropy_step}% save picture pdf with same name as tex file in \tikzfolder
		\input{./TeX/Tikz/Entropy_step.tex}%
	}%
				\caption[Entropic mirror descent for underdetermined linear systems]{%
					Problem \eqref{eq:Plinsys} in \S\ref{sec:simulations:linsys}: underdetermined nonnegative linear system with different sparsity levels.
					Top: convergence in terms of cost.
					Bottom: stepsize magnitudes in a window of the first 200 iterations.
				}%
				\label{fig:LinSys}
			\end{figure}

		%% ░░░░ 5.6 Problems without the Bregman zone assumption ░░░░░░░░░░░░░░░░░░░░░░░░░░░░░░░░░░░░░░░░░░░░░░░░░░░░
		\subsection{Problems without the Bregman zone assumption}\label{sec:simulations:zone}%
			To confirm the wide applicability of \refBY{} in the general setting of \cref{thm:main}, we consider a representative problem in which \cref{ass:zone} is not met and solutions exist only on the boundary of \(\dom\kernel\).
			Specifically, we consider
			\begin{equation}\label{eq:linsysball}
				\minimize_{x\in\overline{C}}
				\tfrac{1}{2}\norm{Ax-b}^2
			\end{equation}
			where \(\overline C=\{x\in\R^n:\norm{x}\leq1\}\) is the unit ball.
			We set \(n=1000\) and randomly generate \(A\in\R^{n\times n}\) together with a candidate solution \(\bar x\in\R^n\) of norm 2, and then set \(b=A\bar x\) so that the solution of \eqref{eq:linsysball} lies on the boundary.
			The ball constraint is enforced by the employment of three different dgfs, namely
			\begin{subequations}\label{eq:zone:dgfs}%
				\begin{align}
				\label{eq:phi:ball}
					\kernel(x)
				& =
					-\sqrt{1-\norm{x}^2},
				\\
				\label{eq:phi:ballfrac}
					\kernel(x)
				& =
					-\tfrac{1}{1-\norm{x}^2},
				\shortintertext{and}
				\label{eq:phi:balllog}
					\kernel(x)
				& =
					-\ln(1-\norm{x}^2)
				\end{align}
			\end{subequations}
			(all extended to \(\infty\) where the expressions are not well defined).
			Each of these functions is not separable and violates \cref{ass:zone}; for \eqref{eq:phi:ballfrac} and \eqref{eq:phi:balllog} this is clear since the domains are not closed, while the case of \eqref{eq:phi:ball} has been shown in \cref{ex:zone}.
			This test is provided as a simple proof of concept of the robustness of the proposed algorithmic framework when less ``benign'' dgfs are employed; for this reason, comparisons with the better suited Euclidean methods are omitted in this case.

			\begin{remark}\label{rem:zonedgfs}%
				All these functions are of the form \(\kernel(x)=\varrho(\norm{x}^2)\) for some convex \(\varrho:\R_+\to\Rinf\) which is twice differentiable and with strictly positive derivative on the interior of its domain.
				It can be easily checked that
				\[
					\nabla\kernel(x)
				=
					2\varrho'(\norm{x}^2)x
				\quad\text{and}\quad
					\nabla^2\kernel(x)
				=
					2\varrho'(\norm{x}^2)\I
					+
					4\varrho''(\norm{x}^2)xx^\top,
				\]
				where \(\I\) is the \(n\times n\) identity matrix, and in particular that each \(\kernel\) is \(\mu\)-strongly convex with
				\[
					\mu=2\inf \varrho'=2\varrho'(0),
				\]
				where the last identity owes to monotony of \(\varrho'\) (due to convexity of \(\varrho\)).
				Thus, evaluating \(\nabla\kernel^*\) reduces to solving a one-dimensional equation, namely
				\[
					\nabla\kernel^*(\xi)
				=
					\tfrac{1}{2\varrho'(t^2)}\xi
				\quad\text{where}\quad
					2t\varrho'(t^2)=\norm{\xi}.
				\]
				This computation is available in closed form for each of the functions \(\kernel\) here considered.
				\hfill\qedsymbol
			\end{remark}

			We then deduce that \(f(x)=\frac{1}{2}\norm{Ax-b}^2\) is globally smooth relative to each \(\kernel\) as in \eqref{eq:zone:dgfs} with modulus
			\(
				L_f^\kernel
			=
				\tfrac{\norm{A}^2}{2\varrho'(0)}
			\),
			where \(\varrho\) is the corresponding radial function as in \cref{rem:zonedgfs}.
			On the other hand, the \emph{local} modulus around any point \(x\in C\) is of the order \(\tfrac{\norm{A}^2}{2\varrho'(\norm{x}^2)}\), which vanishes as \(x\) approaches the boundary.
			This explains the rapid increase of the stepsizes across all instances in \cref{fig:zone}, confirming the ability of adaptive methods to exploit local landscape information without being penalized by global worst-case estimates.

			\begin{figure}[htb]
				\centering
				{%
		\tikzexternalenable
		\pgfkeys{/pgf/images/include external/.code={\includegraphics[width=\linewidth]{#width=\linewidth}}}%
		\tikzsetnextfilename{Zone}% save picture pdf with same name as tex file in \tikzfolder
		\input{./TeX/Tikz/Zone.tex}%
	}
				{%
		\tikzexternalenable
		\pgfkeys{/pgf/images/include external/.code={\includegraphics[width=\linewidth]{#width=\linewidth}}}%
		\tikzsetnextfilename{Zone_step}% save picture pdf with same name as tex file in \tikzfolder
		\input{./TeX/Tikz/Zone_step.tex}%
	}%
				\caption[Problems without the Bregman zone assumption]{%
					Problem \eqref{eq:linsysball} of \S\ref{sec:simulations:zone} using the three dgfs \(\kernel\) as in \eqref{eq:zone:dgfs}, all violating the Bregman zone \cref{ass:zone}.
					Top: convergence in terms of cost.
					Bottom: stepsize magnitudes in a window of the first 200 iterations (normalized by \(L_f^{\kernel}\)).
				}%
				\label{fig:zone}%
			\end{figure}

		%% ░░░░ 5.7 Euclidean problems ░░░░░░░░░░░░░░░░░░░░░░░░░░░░░░░░░░░░░░░░░░░░░░░░░░░░░░░░░░░░░░░░░░░░░░░░░░░░░░░░░░░
		\subsection{Euclidean problems}\label{sec:simulations:Euclidean}%
			We finally assess the proposed adaptive selection strategies in purely Euclidean settings, namely, in which the smooth function has a globally Lipschitz continuous gradient.
			We study two setups, each developed in a dedicated subsection.

			\subsubsection{Bregman vs Euclidean updates}\label{sec:BregmanVSEuclidean}%
				First, we consider a standard lasso problem
				\begin{equation}\label{eq:Euclidean}
					\minimize_{x\in\R^n}
					\tfrac{1}{2}\norm{Ax - b}^2
					+
					\lambda\norm{x}_1,
				\end{equation}
				where \(A\in\R^{m\times n}\), \(b\in\R^m\), and \(\lambda=0.01\) is the regularization parameter promoting sparsity.
				Clearly, \(f(x)=\tfrac{1}{2}\norm{Ax - b}^2\) has \(L_f\)-Lipschitz gradient with \(L_f=\norm{A}^2\).
				For Bregman methods, we select the ``aggressive'' kernel \(\kernel\) \eqref{eq:kernelPoly}, namely the Euclidean \(\j\) augmented by a quadratic function as in \cref{sec:Hessian}:
				\[
					\kernel(x) = \tfrac{1}{4}\norm{x}_2^4 + \tfrac{1}{2}\norm{x}_2^2.
				\]
				By doing so, \(f\) is smooth relative to \(\kernel\) as well, and with same constant \(L_f^{\kernel}=L_f\); however, the difference \(L_f\kernel-f\) is strictly convex, whereas \(L_f\j-f\) is not.
				This test is meant to compare the performance of Bregman vs Euclidean updates.
				Differently from other simulations, due to numerical issues close to solutions the tolerance of \(\D_\kernel(x^k,x^{k-1})\) in the termination criterion is set to \(10^{-9}\).

				\begin{figure}[h]
					\centering
					{%
		\tikzexternalenable
		\pgfkeys{/pgf/images/include external/.code={\includegraphics[width=\linewidth]{#width=\linewidth}}}%
		\tikzsetnextfilename{BregmanLasso}% save picture pdf with same name as tex file in \tikzfolder
		\input{./TeX/Tikz/BregmanLasso.tex}%
	}
					{%
		\tikzexternalenable
		\pgfkeys{/pgf/images/include external/.code={\includegraphics[width=\linewidth]{#width=\linewidth}}}%
		\tikzsetnextfilename{BregmanLasso_step}% save picture pdf with same name as tex file in \tikzfolder
		\input{./TeX/Tikz/BregmanLasso_step.tex}%
	}%
					\caption[Bregman vs Euclidean methods]{%
						Bregman vs Euclidean methods as in \S\ref{sec:BregmanVSEuclidean}:
						Random lasso using quartic kernel \(\kernel\) \eqref{eq:kernelPoly} for Bregman methods.
						Top: convergence in terms of cost.
						Bottom: stepsizes normalized by \(L_f^{\kernel}\) in a window of the first 200 iterations.%
					}%
					\label{fig:LassoDh}
				\end{figure}

				As shown in the results of \cref{fig:LassoDh}, Bregman updates seem to outperform purely Euclidean proximal gradient steps.
				While surprising, this phenomenon can be attributed to the ``higher curvature'' of \(L_f\kernel-f\) with respect to that of \(L_f\j-f\), a behavior that we find worthy of future investigations.

			\subsubsection{Conservatism when \texorpdfstring{\(\kernel=\j\)}{phi=j}}\label{sec:conservatism}%
				As discussed in \cref{sec:results}, the Bregman analysis investigated here introduces some conservatism; that is, when specialized to the Euclidean kernel \(\kernel=\j\), \refBY{} and \refFNEa{} reduce to \emph{dampened} variants of \cite[\adaPG]{latafat2024adaptive} and \cite[\adaPG$^{1,\frac{1}{2}}$]{latafat2024convergence}.
				This second test investigates how the choice \(\kernel=\j\) penalizes the performance with respect to \adaPG{} and \adaPGq.

				Our experiments are based on the Julia code provided in \cite{latafat2024adaptive}, and the test problems are sourced from the LIBSVM dataset \cite{chang2011libsvm}.
				We added our methods into the original test framework and conducted numerical experiments on the same problems.
				As \cref{fig:Euclidean} demonstrates, although the discussed differences do have some impact on the algorithms' performance, the outcomes remain acceptable.
				In fact, to some extent there doesn't appear to be a clear winner.
				This counterintuitive observation can be attributed to the fact that a small stepsize at an iteration \(k\) may trigger a larger stepsize at the next one.
				We believe that also this aspect is an interesting direction for future research.

				\begin{figure}[p]
					\centering
					\fbox{%
						\parbox{0.88\linewidth}{%
							{%
		\tikzexternalenable
		\pgfkeys{/pgf/images/include external/.code={\includegraphics[width=\linewidth]{#width=\linewidth}}}%
		\tikzsetnextfilename{Euc_sparse_logreg}% save picture pdf with same name as tex file in \tikzfolder
		\input{./TeX/Tikz/Euc_sparse_logreg.tex}%
	}
							{%
		\tikzexternalenable
		\pgfkeys{/pgf/images/include external/.code={\includegraphics[width=\linewidth]{#width=\linewidth}}}%
		\tikzsetnextfilename{Euc_sparse_logreg_step}% save picture pdf with same name as tex file in \tikzfolder
		\input{./TeX/Tikz/Euc_sparse_logreg_step.tex}%
	}%
						}%
					}%

					\smallskip

					\fbox{%
						\parbox{0.88\linewidth}{%
							{%
		\tikzexternalenable
		\pgfkeys{/pgf/images/include external/.code={\includegraphics[width=\linewidth]{#width=\linewidth}}}%
		\tikzsetnextfilename{Euc_cubic}% save picture pdf with same name as tex file in \tikzfolder
		\input{./TeX/Tikz/Euc_cubic.tex}%
	}
							{%
		\tikzexternalenable
		\pgfkeys{/pgf/images/include external/.code={\includegraphics[width=\linewidth]{#width=\linewidth}}}%
		\tikzsetnextfilename{Euc_cubic_step}% save picture pdf with same name as tex file in \tikzfolder
		\input{./TeX/Tikz/Euc_cubic_step.tex}%
	}%
						}%
					}%

					\smallskip

					\fbox{%
						\parbox{0.88\linewidth}{%
							{%
		\tikzexternalenable
		\pgfkeys{/pgf/images/include external/.code={\includegraphics[width=\linewidth]{#width=\linewidth}}}%
		\tikzsetnextfilename{Euc_lasso}% save picture pdf with same name as tex file in \tikzfolder
		\input{./TeX/Tikz/Euc_lasso.tex}%
	}
							{%
		\tikzexternalenable
		\pgfkeys{/pgf/images/include external/.code={\includegraphics[width=\linewidth]{#width=\linewidth}}}%
		\tikzsetnextfilename{Euc_lasso_step}% save picture pdf with same name as tex file in \tikzfolder
		\input{./TeX/Tikz/Euc_lasso_step.tex}%
	}%
						}%
					}%

					\caption[Conservatism when \(\kernel=\j\)]{%
						Conservatism when \(\kernel=\j\) as in \S\ref{sec:conservatism}:
						Comparisons with Euclidean adaptive methods on sparse logistic regression (top rows), cubic regularization (middle rows), and lasso problems (bottom rows).
					}
					\label{fig:Euclidean}%
				\end{figure}

	%% ██ 6. Conclusions █████████████████████████████████████████████████████████████████████████████████████████████████
	\section{Conclusions}
		This paper introduced new adaptive stepsize strategies for Bregman proximal gradient algorithms that eliminate the need for traditional backtracking procedures.
		The proposed methods determine stepsizes dynamically based solely on certain local curvature estimates derived from gradients at the current and previous iterations.
		This approach enables large stepsizes, often several orders of magnitude larger than their constant stepsize counterparts, leading to fast convergence while maintaining theoretical convergence guarantees.
		Notably, the theoretical analysis operates under minimal assumptions, requiring only local relative smoothness for the differentiable term, and local strong convexity for the Bregman kernel, rather than their global counterparts, and is thus also agnostic to any such global moduli.
		A key technical step in our analysis is the employment of a Bregman generalization of Young's inequality, which, despite its simplicity, proves essential to the analysis, and is interesting in its own right.

		When specialized to \(\kernel=\j\), the proposed algorithms recover the Euclidean counterparts in \cite{latafat2024adaptive,latafat2024convergence} up to some slight conservatism.
		Regardless, as shown in the simulations, the flexibility to accommodate arbitrary 1-coercive and Legendre kernels \(\kernel\) has remarkable practical advantages even when \(\dom\kernel=\R^n\).

		Some important theoretical questions remain open.
		Our numerical evidence highlights that the proposed adaptive methods generate stepsizes that stay bounded away from zero, even when approaching boundary points of the domain of the kernel function, and fluctuate around values attained by aggressive linesearch routines.
		This trend is consistent with observations documented in previous studies in the Euclidean setting, but theoretical confirmations in the more general Bregman setup currently do not have a definite answer.
		We believe that further advances in this direction will require new insights into the Bregman curvature estimate \(\Lambda_{k,\delta}\) of \eqref{eq:Lamk}, possibly enabling the derivation of uniform lower bounds on \(\gamma_k\) within \cref{thm:Lamkto1}.

	%% ██ A. Omitted proofs ██████████████████████████████████████████████████████████████████████████████████████████████
	\appendix
	\section{Omitted proofs}
		%% ░░░░ Proof of Remark 2.8 ░░░░░░░░░░░░░░░░░░░░░░░░░░░░░░░░░░░░░░░░░░░░░░░░░░░░░░░░░░░░░░░░░░░░░░░░░░░░░░░░░░░░░░
		\phantomsection
		\addcontentsline{toc}{subsection}{Proof of Remark \ref*{rem:eps}}%
		\label{proof:rem:eps}%
		\begin{proof}[\textbf{Proof of \cref{rem:eps}}]%
			As commented in \cref{foot:noFNEa}, it suffices to consider the iterates generated by \refBY.
			Observe that the validity of \cref{thm:main} is unaffected by the \(\epsilon\)-perturbation.
			Indeed, as long as \(\epsilon\) is chosen small (specifically, strictly smaller than \(\frac{2-\sqrt{2}}{2}\)), the sequence \(\rho_{k+1}=(1-\epsilon)\sqrt{1+\rho_k}\) converges to a value strictly larger than one, and all the proofs verbatim apply (most crucially, the second claim in the proof of \cref{thm:xinC}).
			Let \((x^k)_{k\in K}\) be a subsequence converging to a solution \(x^\star\), ensured to exist by \cref{thm:main:zone}.
			From the update modification, and since \(P_k(x^\star)\geq0\), it follows from \eqref{eq:UkBY:leq} and \eqref{eq:UkBYdt:leq} that
			\[
				\U_{k+1}(x^\star)
			\leq
				\U_k(x)
				-
				\epsilon\gamma_k
				P_{k-1}
				-
				\tfrac{\epsilon}{2}\D_\kernel(x^k,x^{k-1})
			\]
			holds for any \(k\), where \(P_{k-1}\coloneqq P_{k-1}(x^\star)=\varphi(x^{k-1})-\min_{\overline C}\varphi\geq0\).
			In particular, \(\gamma_kP_{k-1}\) and \(\D_\kernel(x^k,x^{k-1})\) both vanish.
			Therefore, so does \(\U_{k+1}(x^\star)\) as \(K\ni k\to\infty\), and since the whole sequence is convergent its limit must be zero.
			From the inequality \(\D_\kernel(x^\star,x^k)\leq\U_k(x^\star)\) we conclude that \(\D_\kernel(x^\star,x^k)\) too vanishes, hence that \(x^k\to x^\star\) by \cref{ass:zone}.
		\end{proof}

		%% ░░░░ Proof of Lemma 3.1 ░░░░░░░░░░░░░░░░░░░░░░░░░░░░░░░░░░░░░░░░░░░░░░░░░░░░░░░░░░░░░░░░░░░░░░░░░░░░░░░░░░░░░░░
		\phantomsection
		\addcontentsline{toc}{subsection}{Proof of Lemma \ref*{thm:eq}}%
		\label{proof:thm:eq}%
		\begin{proof}[\textbf{Proof of \cref{thm:eq}}]%
			Based on the subgradient characterization \eqref{eq:subgradvarphi} and the definition of \(\ell_k\) in \eqref{eq:lk}, we have
			\begin{align}
			\nonumber
				0
			={} &
				\varphi(x^{k-1})
				-
				\varphi(x^k)
				-
				\tfrac{1}{\gamma_k}
				\innprod{H_k(x^{k-1})-H_k(x^k)}{x^{k-1}-x^k}
				-
				\tD_\varphi(x^{k-1},x^k)
			\\
			={} &
				P_{k-1}(x)
				-
				P_k(x)
				-
				\tfrac{1-\gamma_k\ell_k}{\gamma_k}
				\DD_{\kernel}(x^k,x^{k-1})
				-
				\tD_\varphi(x^{k-1},x^k).
			\label{eq:thetk}
			\end{align}
			Furthermore, by leveraging the subgradient characterization \eqref{eq:subgrad} and applying the three point identity of \cref{thm:3p},
			\begin{align}
			\nonumber
				0
			={} &
				g(x)-g(x^{k+1})-\innprod{\tilde{\nabla}g(x^{k+1})}{x-x^{k+1}}-\tD_g(x,x^{k+1})
			\\
			\nonumber
			={} &
				g(x)-g(x^{k+1})
				+
				\innprod{\nabla f(x^k)}{x- x^{k+1}}
			\\
			\nonumber
			&
				-
				\tfrac{1}{\gamma_{k+1}}\innprod{\nabla\kernel(x^k)-\nabla\kernel(x^{k+1})}{x-x^{k+1}}
				-
				\tD_g(x,x^{k+1})
			\\
			\nonumber
			={} &
				g(x)-g(x^{k+1})
				+
				\underbracket[0.5pt]{
					\innprod{\nabla f(x^k)}{x- x^{k+1}}
				}_{\text{(A)}}
				-
				\tD_g(x,x^{k+1})
			\\
			&
				+
				\tfrac{1}{\gamma_{k+1}}\D_{\kernel}(x,x^k)
				-
				\tfrac{1}{\gamma_{k+1}}\D_{\kernel}(x,x^{k+1})
				-
				\tfrac{1}{\gamma_{k+1}}\D_{\kernel}(x^{k+1},x^k).
			\label{eq:main}
			\end{align}
			We next proceed to expand the term (A) as
			\begin{align*}
				\text{(A)}
			={} &
				\innprod{\nabla f(x^k)}{x-x^k}
				-
				\innprod{\nabla f(x^k)}{x^{k+1}-x^k}
			\\
			={} &
				\innprod{\nabla f(x^k)}{x-x^k}
				+
				\tfrac{1}{\gamma_k}
				\innprod{H_k(x^{k-1})-\nabla\kernel(x^k)}{x^{k+1}-x^k}
			\\
			&
				+
				\tfrac{1}{\gamma_k}
				\innprod{H_k(x^{k-1})-H_k(x^k)}{x^k-x^{k+1}}
			\\
			={} &
				\Bigl[
					f(x)-f(x^k)-\D_f(x,x^k)
				\Bigr]
				+
				\Bigl[
					g(x^{k+1})-g(x^k)-\tD_g(x^{k+1},x^k)
				\Bigr]
			\\
			&
				+
				\tfrac{1}{\gamma_{k+1}}
				B_{k+1},
			\end{align*}
			which combined with \eqref{eq:main} gives
			\begin{align*}
				0
			={} &
				-P_k(x)
				+
				\tfrac{1}{\gamma_{k+1}}\D_{\kernel}(x,x^k)
				-
				\tfrac{1}{\gamma_{k+1}}\D_{\kernel}(x,x^{k+1})
				-
				\tfrac{1}{\gamma_{k+1}}\D_{\kernel}(x^{k+1},x^k)
			\\
			&
				-\D_f(x,x^k)
				-\tD_g(x^{k+1},x^k)
				+
				\tfrac{1}{\gamma_{k+1}}
				B_{k+1}
				-
				\tD_g(x,x^{k+1}).
			\end{align*}
			We may now add \eqref{eq:main} to \eqref{eq:thetk} scaled by \(\vartheta_{k+1}\) and multiply everything by \(\gamma_{k+1}\) to obtain
			\begin{align*}
				0
			={} &
				-\gamma_{k+1} P_k(x)
				+
				\D_{\kernel}(x,x^k)
				-
				\D_{\kernel}(x,x^{k+1})
				-
				\D_{\kernel}(x^{k+1},x^k)
			\\
			&
				-
				\gamma_{k+1}\D_f(x,x^k)
				-
				\gamma_{k+1}\tD_g(x^{k+1},x^k)
				+
				B_{k+1}
				-
				\gamma_{k+1}\tD_g(x,x^{k+1})
			\\
			&
				+
				\gamma_{k+1}\vartheta_{k+1}\left(
					P_{k-1}(x)
					-
					P_k(x)
					-
					\tfrac{1-\gamma_k\ell_k}{\gamma_k}
					\DD_{\kernel}(x^k,x^{k-1})
					-
					\tD_\varphi(x^{k-1},x^k)
				\right).
			\end{align*}
			After suitably rearranging, the claimed identity is obtained.

			The inequality follows by neglecting the terms between curly brackets, and further using the identity \(\D_{\kernel}(x^k,x^{k-1})=\frac{\alpha_k}{1+\alpha_k}\DD_{\kernel}(x^k,x^{k-1})\).
		\end{proof}

		%% ░░░░ Proof of Lemma 4.1 ░░░░░░░░░░░░░░░░░░░░░░░░░░░░░░░░░░░░░░░░░░░░░░░░░░░░░░░░░░░░░░░░░░░░░░░░░░░░░░░░░░░░░░░
		\phantomsection
		\addcontentsline{toc}{subsection}{Proof of Lemma \ref*{thm:descent}}%
		\label{proof:thm:descent}%
		\begin{proof}[\textbf{Proof of \cref{thm:descent}}]%
			It suffices to prove the claim for \refBY; the case of \refFNEa{} under the needed assumptions follows by simply replacing \(\U_k\gets\U^{\alpha}_k\).
			The assumption on \(x\) ensures that \(P_k(x)\geq0\) holds for all \(k\), implying both that \(\U_k(x)\geq0\) and the claimed monotonic decrease of \(\bigl(\U_k(x)\bigr)_{k\in\N}\) with finite limit by virtue of \cref{thm:BY:leq,thm:FNEa:leq}.
			More precisely, one has that
			\[
				0
			\leq
				\U_{k+1}(x)
			\leq
				\U_k(x)
				-
				\gamma_k(1+q\hat\rho_k-q\hat\rho_{k+1}\rho_{k+1})P_{k-1}(x),
			\]
			where \(q=1\) for \refBY{} and \(q=\frac{1+\alpha}{2\alpha}\) for \refFNEa{}.
			A telescoping argument yields that
			\begin{align*}
				P_K^{\min}(x)
				\sum_{k=1}^K
				\gamma_k(1+q\hat\rho_k-q\hat\rho_{k+1}\rho_{k+1})
			\leq{} &
				\sum_{k=1}^K
				\gamma_k(1+q\hat\rho_k-q\hat\rho_{k+1}\rho_{k+1})P_{k-1}(x)\\
			\leq{} &
				\U_1(x) - \U_{K+1}(x)\\
			\leq{} &
				\U_1(x) - \gamma_{K+1}(1+q\hat\rho_{K+1})P_K^{\rm min}(x),
			\end{align*}
			hence that
			\begin{align*}
				P_K^{\min}(x)
			\leq{} &
				\frac{\U_1(x)
					}{
					\sum_{k=1}^K(\gamma_k + q\gamma_k\hat\rho_k- q\gamma_{k+1}\hat\rho_{k+1}) + \gamma_{K+1} + q\gamma_{K+1}\hat\rho_{K+1}
					}\\
			\leq{} &
				\frac{\U_1(x)
					}{
					\gamma_1\hat\rho_1
					+
					\sum_{k=1}^{K+1}\gamma_k
					}.
			\end{align*}
			Since \(\U_1(x)\leq\U_0(x)\), the claimed inequality follows.
		\end{proof}

	%% ██ References █████████████████████████████████████████████████████████████████████████████████████████████████████
	\small
	\phantomsection
	\addcontentsline{toc}{section}{References}
	\bibliographystyle{plain}
	\bibliography{references.bib}

\end{document}

%% file: Preamble.tex
	\PassOptionsToPackage{dvipsnames,table}{xcolor}
	\usepackage{amsmath}
	\usepackage{amssymb}
	\usepackage{amsthm}
	\usepackage[USenglish]{babel}
	\usepackage{comment}
	\usepackage{eqparbox} % for widestL enumitem option
	\usepackage{enumitem}
	\usepackage{etoolbox}
	\usepackage[T1]{fontenc}
	\usepackage{mathtools}
	\usepackage{pifont}
	\usepackage{xcolor}
	\usepackage{xpatch}
	\usepackage{xstring}
	\usepackage{hyperref}
	\usepackage[capitalize,nameinlink]{cleveref} % don't use texlive 2025!

	\counterwithin{figure}{subsection}
	\makeatletter
		\xpatchcmd{\tableofcontents}{\contentsname \@mkboth}{\large\contentsname \@mkboth}{}{}
	\makeatother
	\AtBeginDocument{%
		\addtocontents{toc}{\footnotesize}%
	}

	\SetEnumitemKey{widestL}{
		labelwidth=\eqboxwidth{listlabel@\EnumitemId},
		leftmargin=!,
		format=\mylistlabel,
	}
	\newcommand\mylistlabel[2][l]{\eqmakebox[listlabel@\EnumitemId][#1]{#2}}

	\setlist[enumerate,1]{
		label=\text{\upshape{(\roman*)}},
	}
	\setlist[itemize,2]{
		wide,
		label=\raisebox{0.3ex}{\tiny$\bullet$},
		leftmargin = 0pt,
		labelindent = 0pt,
	}
	\setlist{
		widestL,
		topsep=1ex,
		itemsep=0pt,
	}

	\hypersetup{%
		breaklinks,
		hypertexnames = true,
		plainpages    = false,
		colorlinks = true,
		urlcolor   = MidnightBlue,
		linkcolor  = MidnightBlue,
		citecolor  = ForestGreen,
	}
	\numberwithin{equation}{section}
	\newcommand{\createthmlists}[1]{%
		\newlist{#1enumerate}{enumerate}{1}%
		\setlist[#1enumerate,1]{%
			label = \upshape{(\roman*)},%
			ref = \text{\thetheorem(\roman*)},%
		}%
		\crefalias{#1enumeratei}{#1}%
		\newlist{#1enumerateq}{enumerate}{1}%
		\setlist[#1enumerateq,1]{%
			label = \upshape{(\alph*)},%
			ref = \text{\thetheorem(\alph*)},%
		}%
		\crefalias{#1enumerateqi}{#1}%
		\AtBeginEnvironment{#1}{%
			\expandafter\let\expandafter\enumerate\csname #1enumerate\endcsname
			\expandafter\let\expandafter\endenumerate\csname end#1enumerate\endcsname
			\expandafter\let\expandafter\enumerateq\csname #1enumerateq\endcsname
			\expandafter\let\expandafter\endenumerateq\csname end#1enumerateq\endcsname
		}{}{}%
	}%

	\theoremstyle{plain}
	\newtheorem{theorem}{Theorem}[section]
		\createthmlists{theorem}

	\newtheorem{claim}{Claim}

	\newtheorem{assumption}[theorem]{Assumption}
		\createthmlists{assumption}
		\Crefname{assumption}{Assumption}{Assumptions}
	\newtheorem{corollary}[theorem]{Corollary}
		\createthmlists{corollary}
		\Crefname{corollary}{Corollary}{Corollaries}
	
	\newtheorem{fact}[theorem]{Fact}
		\createthmlists{fact}
		\Crefname{fact}{Fact}{Facts}
	\newtheorem{lemma}[theorem]{Lemma}
		\createthmlists{lemma}

	\theoremstyle{definition}
	\newtheorem{example}[theorem]{Example}
	\newtheorem{remark}[theorem]{Remark}
		\createthmlists{remark}

	\newcommand{\algnamefont}{\sffamily}

	\newcommand{\BY}{{\algnamefont B-adaPG}}
	\newcommand{\FNEa}{{\algnamefont B-adaPG\texorpdfstring{\(_\alpha\)}{a}}}

	\newcommand{\ABPG}{{\algnamefont ABPG}}
	\newcommand{\ABPGg}{\ABPG{\algnamefont-g}}
	\newcommand{\adaPG}{{\algnamefont adaPG}}
	\newcommand{\adaPGq}{\adaPG{\algnamefont\(^{1,\frac{1}{2}}\)}}
	\newcommand{\BaGRA}{{\algnamefont BaGRAAL}}
	\newcommand{\BPG}{{\algnamefont B-PG}}
	\newcommand{\BPGls}{\BPG{\algnamefont-ls}}
	\newcommand{\MDPolyak}{{\algnamefont MD-Polyak}}
	\newcommand{\PG}{{\algnamefont PG}}
	\newcommand{\PGls}{\PG{\algnamefont-ls}}

	\newcommand{\refBPGls}{\texorpdfstring{\hyperref[alg:BPGls]{\BPGls}}{\BPGls}}
	\newcommand{\refBY}{\texorpdfstring{\hyperref[alg:BY]{\BY}}{\BY}}
	\newcommand{\refFNEa}{\texorpdfstring{\hyperref[alg:FNEa]{\FNEa}}{\FNEa}}

	\DeclareSymbolFont{varmathbb}{U}{txmia}{m}{it}
	\DeclareMathSymbol{\mathbbcharN}{\mathord}{varmathbb}{142}
	\DeclareMathSymbol{\mathbbcharR}{\mathord}{varmathbb}{146}

	\makeatletter
		\newcommand{\I}{\mathop{}\!{\operator@font I}}
	\makeatother
	\newcommand{\K}{\mathcal{K}}
	\newcommand{\N}{\mathbbcharN}
	\newcommand{\R}{\mathbbcharR}
	\newcommand{\Rinf}{\overline{\R}}
	\newcommand{\U}{\mathcal{U}}
	\newcommand{\xmark}{\ding{55}}%

	\DeclarePairedDelimiter\norm{\lVert}{\rVert}
	\DeclarePairedDelimiter\abs{\lvert}{\rvert}
	\newcommand{\innprod}[2]{\langle#1,#2\rangle}

	\DeclareMathOperator*{\argmin}{arg\,min}
	\DeclareMathOperator{\diag}{diag}
	\DeclareMathOperator{\dom}{dom}
	\DeclareMathOperator{\epi}{epi}

	\DeclareMathOperator{\indicator}{\delta}
	\DeclareMathOperator{\interior}{int}
	\DeclareMathOperator{\KL}{KL}
	\DeclareMathOperator*{\minimize}{minimize}
	\DeclareMathOperator{\symm}{Sym}
	\makeatletter
		\newcommand{\prox}{\mathop{}\!\overleftarrow{\operator@font prox}^{\kernel}}
	\makeatother

	\newcommand{\kernel}{\phi}
	\DeclareMathOperator{\D}{D}				% Bregman distance
	\DeclareMathOperator{\DD}{\Delta}		% symmetrized Bregman distance
	\DeclareMathOperator{\tD}{\tilde{D}}	% Bregman distance with subgradient

	\DeclareFontFamily{U}{boondoxuprscr}{\skewchar \font =45}
	\DeclareFontShape{U}{boondoxuprscr}{m}{n}{
		<-> BOONDOXUprScr-Regular}{}
	\DeclareFontShape{U}{boondoxuprscr}{b}{n}{
		<-> BOONDOXUprScr-Bold}{}
	\DeclareMathAlphabet{\mathboondoxo}{U}{boondoxuprscr}{m}{n}
	\SetMathAlphabet{\mathboondoxo}{bold}{U}{boondoxuprscr}{b}{n}
	\renewcommand\j{\mathboondoxo{j}}

%% file: references.bib
@article{ahookhosh2021bregman,
	author		= {Ahookhosh, Masoud and Themelis, Andreas and Patrinos, Panagiotis},
	title		= {A {B}regman Forward-Backward Linesearch Algorithm for Nonconvex Composite Optimization: Superlinear Convergence to Nonisolated Local Minima},
	year		= {2021},
	journal		= {SIAM Journal on Optimization},
	volume		= {31},
	pages		= {653-685},
	number		= {1},
	doi			= {10.1137/19M1264783},
}

@article{aujol2025stochastic,
	author		= {Aujol, Jean-François and Bigot, Jérémie and Castera, Camille},
	title		= {Stochastic Adaptive Gradient Descent Without Descent},
	year		= {2025},
	journal		= {\arXivLink{2509.14969}},
}

@article{barzilai1988two,
	author		= {Barzilai, Jonathan and Borwein, Jonathan M.},
	title		= {Two-Point Step Size Gradient Methods},
	year		= {1988},
	journal		= {IMA Journal of Numerical Analysis},
	volume		= {8},
	pages		= {141-148},
	number		= {1},
}

@article{bauschke1997legendre,
	author		= {Bauschke, Heinz H. and Borwein, Jonathan M.},
	title		= {Legendre functions and the method of random {B}regman projections.},
	year		= {1997},
	journal		= {Journal of Convex Analysis},
	volume		= {4},
	pages		= {27-67},
	number		= {1},
	publisher	= {Heldermann Verlag},
}

@book{bauschke2017convex,
	author		= {Bauschke, Heinz H. and Combettes, Patrick L.},
	title		= {Convex Analysis and Monotone Operator Theory in {H}ilbert Spaces},
	year		= {2017},
	publisher	= {Springer},
	isbn		= {978-3-319-48310-8},
}

@article{bauschke2017descent,
	author		= {Bauschke, Heinz H. and Bolte, Jérôme and Teboulle, Marc},
	title		= {A Descent Lemma Beyond {L}ipschitz Gradient Continuity: First-Order Methods Revisited and Applications},
	year		= {2017},
	journal		= {Mathematics of Operations Research},
	volume		= {42},
	pages		= {330--348},
	number		= {2},
	publisher	= {INFORMS},
	doi			= {10.1287/moor.2016.0817},
}

@article{bolte2018first,
	title		= {First order methods beyond convexity and {L}ipschitz gradient continuity with applications to quadratic inverse problems},
	author		= {Bolte, Jérôme and Sabach, Shoham and Teboulle, Marc and Vaisbourd, Yakov},
	journal		= {SIAM Journal on Optimization},
	volume		= {28},
	number		= {3},
	pages		= {2131--2151},
	year		= {2018},
	publisher	= {SIAM},
	doi			= {doi.org/10.1137/17M1138558},
}

@article{chang2011libsvm,
	author		= {Chang, Chih-Chung and Lin, Chih-Jen},
	title		= {{LIBSVM}: A library for support vector machines},
	year		= {2011},
	journal		= {ACM Transactions on Intelligent Systems and Technology (TIST)},
	volume		= {2},
	pages		= {1--27},
	number		= {3},
	publisher	= {Acm New York, NY, USA},
}

@article{chen1993convergence,
	author		= {Chen, Gong and Teboulle, Marc},
	title		= {Convergence Analysis of a Proximal-Like Minimization Algorithm Using {B}regman Functions},
	year		= {1993},
	journal		= {SIAM Journal on Optimization},
	volume		= {3},
	pages		= {538-543},
	number		= {3},
	doi			= {10.1137/0803026},
}

@article{fujiki2026approximate,
	author		= {Fujiki, Kiwamu and Takahashi, Shota and Takeda, Akiko},
	title		= {Approximate {B}regman proximal gradient algorithm with variable metric {A}rmijo--{W}olfe line search},
	year		= {2026},
	journal		= {\arXivLink{2510.06615}},
}

@article{grimmer2025composing,
	author		= {Grimmer, Benjamin and Shu, Kevin and Wang, Alex L.},
	title		= {Composing Optimized Stepsize Schedules for Gradient Descent},
	year		= {2025},
	journal		= {Mathematics of Operations Research},
	doi			= {10.1287/moor.2024.0764},
}

@article{hanzely2021accelerated,
  title   		= {Accelerated {B}regman Proximal Gradient Methods for Relatively Smooth Convex Optimization},
  author  		= {Hanzely, Filip and Richtárik, Peter and Xiao, Lin},
  journal 		= {Computational Optimization and Applications},
  year    		= {2021},
  volume  		= {79},
  number  		= {2},
  pages   		= {405--440},
  doi     		= {10.1007/s10589-021-00273-8}
}

@article{jang2026alia,
	author		= {Jang, Uijeong and Sun, Kaizhao and Yin, Wotao and Ryu, Ernest K.},
	title		= {{ALiA}: Adaptive Linearized {ADMM}},
	year		= {2026},
	journal		= {\arXivLink{2602.15000}},
}

@article{ji2026stochastic,
	author		= {Ji, Yao and Lan, Guanghui},
	title		= {Stochastic Auto-conditioned Fast Gradient Methods with Optimal Rates},
	year		= {2026},
	journal		= {\arXivLink{2604.06525}},
}

@article{latafat2024adaptive,
	author		= {Latafat, Puya and Themelis, Andreas and Stella, Lorenzo and Patrinos, Panagiotis},
	title		= {Adaptive proximal algorithms for convex optimization under local {L}ipschitz continuity of the gradient},
	year		= {2024},
	journal		= {Mathematical Programming},
	doi			= {10.1007/s10107-024-02143-7},
}

@inproceedings{latafat2024convergence,
	author		= {Latafat, Puya and Themelis, Andreas and Patrinos, Panagiotis},
	title		= {On the convergence of adaptive first order methods: {P}roximal gradient and alternating minimization algorithms},
	year		= {2024},
	volume		= {242},
	series		= {Proceedings of Machine Learning Research},
	pages		= {197--208},
	booktitle	= {Proceedings of the 6th Annual Learning for Dynamics \& Control Conference},
	publisher	= {PMLR},
}

@phdthesis{li2018learning,
	author		= {Li, Yen-Huan},
	title		= {Learning without Smoothness and Strong Convexity},
	year		= {2018},
	school		= {EPFL},
}

@article{li2025simple,
	author		= {Li, Tianjiao and Lan, Guanghui},
	title		= {A simple uniformly optimal method without line search for convex optimization},
	year		= {2025},
	journal		= {Mathematical Programming},
	pages		= {1--38},
	publisher	= {Springer},
}

@article{lu2018relatively,
	author		= {Lu, Haihao and Freund, Robert M. and Nesterov, Yurii},
	title		= {Relatively Smooth Convex Optimization by First-Order Methods, and Applications},
	year		= {2018},
	journal		= {SIAM Journal on Optimization},
	volume		= {28},
	pages		= {333-354},
	number		= {1},
}

@inproceedings{malitsky2020adaptive,
	author		= {Malitsky, Yura and Mishchenko, Konstantin},
	title		= {Adaptive Gradient Descent without Descent},
	year		= {2020},
	volume		= {119},
	pages		= {6702--6712},
	booktitle	= {Proceedings of the 37th International Conference on Machine Learning},
}

@article{malitsky2020golden,
	author		= {Malitsky, Yura},
	title		= {Golden ratio algorithms for variational inequalities},
	year		= {2020},
	journal		= {Mathematical Programming},
	volume		= {184},
	pages		= {383--410},
	number		= {1},
	publisher	= {Springer},
}

@inproceedings{malitsky2024adaptive,
	author		= {Malitsky, Yura and Mishchenko, Konstantin},
	title		= {Adaptive Proximal Gradient Method for Convex Optimization},
	year		= {2024},
	volume		= {37},
	pages		= {100670--100697},
	booktitle	= {Advances in NeurIPS},
}

@article{malitsky2026entropic,
	author		= {Malitsky, Yura and Posch, Alexander},
	title		= {Entropic Mirror Descent for Linear Systems: {P}olyak's Stepsize and Implicit Bias},
	year		= {2026},
	journal		= {\arXivLink{2505.02614}},
}

@article{nesterov1983method,
	author		= {Nesterov, Yurii},
	journal		= {Soviet Mathematics Doklady},
	title		= {A method of solving a convex programming problem with convergence rate {{\(O(1/k^2)\)}}},
	volume		= {27},
	year		= {1983}
}

@article{nilsson2025symmetry,
	author		= {Nilsson, Max and Giselsson, Pontus},
	title		= {The Symmetry Coefficient of Positively Homogeneous Functions},
	year		= {2025},
	journal		= {Optimization Letters},
}

@inproceedings{oikonomidis2024adaptive,
	author		= {Oikonomidis, Konstantinos and Laude, Emanuel and Latafat, Puya and Themelis, Andreas and Patrinos, Panagiotis},
	title		= {Adaptive Proximal Gradient Methods Are Universal Without Approximation},
	year		= {2024},
	volume		= {235},
	pages		= {38663--38682},
	booktitle	= {Proceedings of the 41st ICML},
	publisher	= {PMLR},
}

@inproceedings{ou2024safeguarding,
	author		= {Ou, Hongjia and Themelis, Andreas},
	title		= {Safeguarding adaptive methods: {G}lobal convergence of {B}arzilai-{B}orwein and other stepsize choices},
	year		= {2024},
	pages		= {2802-2807},
	booktitle	= {10th International Conference on Control, Decision and Information Technologies (CoDIT)},
	doi			= {10.1109/CoDIT62066.2024.10708282},
}

@article{polyak1969minimization,
	author		= {Polyak, Boris T.},
	title		= {Minimization of unsmooth functionals},
	year		= {1969},
	journal		= {USSR Computational Mathematics and Mathematical Physics},
	volume		= {9},
	pages		= {14-29},
	number		= {3},
	issn		= {0041-5553},
	doi			= {10.1016/0041-5553(69)90061-5},
}

@article{rebegoldi2018bregman,
	author		= {Rebegoldi, Simone and Bonettini, Silvia and Prato, Marco},
	title		= {A {B}regman inexact linesearch--based forward--backward algorithm for nonsmooth nonconvex optimization},
	year		= {2018},
	journal		= {Journal of Physics: Conference Series},
	volume		= {1131},
	pages		= {012013},
	number		= {1},
	publisher	= {IOP Publishing},
	doi			= {10.1088/1742-6596/1131/1/012013},
}

@book{rockafellar1970convex,
	author		= {Rockafellar, Ralph T.},
	title		= {Convex Analysis},
	year		= {1970},
	publisher	= {Princeton University Press},
}

@article{solodov2000inexact,
	author		= {Solodov, Mikhail V. and Svaiter, Benar F.},
	title		= {An Inexact Hybrid Generalized Proximal Point Algorithm and Some New Results on the Theory of {B}regman Functions},
	year		= {2000},
	journal		= {Mathematics of Operations Research},
	volume		= {25},
	pages		= {214--230},
	number		= {2},
	publisher	= {INFORMS},
	doi			= {10.1287/moor.25.2.214.12222},
}

@article{suh2025adaptive,
	author		= {Suh, Jaewook J. and Ma, Shiqian},
	title		= {An Adaptive and Parameter-Free {N}esterov's Accelerated Gradient Method for Convex Optimization},
	year		= {2025},
	journal		= {\arXivLink{2505.11670}},
}

@article{takahashi2025approximate,
	author		= {Takahashi, Shota and Takeda, Akiko},
	title		= {Approximate {B}regman proximal gradient algorithm for relatively smooth nonconvex optimization},
	year		= {2025},
	journal		= {Computational Optimization and Applications},
	volume		= {90},
	pages		= {227--256},
	number		= {1},
	publisher	= {Springer},
}

@article{tam2023bregman,
	author		= {Tam, Matthew K. and Uteda, Daniel J.},
	title		= {{B}regman-golden ratio algorithms for variational inequalities},
	year		= {2023},
	journal		= {Journal of Optimization Theory and Applications},
	volume		= {199},
	pages		= {993--1021},
	number		= {3},
	publisher	= {Springer},
}

@article{wang2022bregman,
	author		= {Wang, Xianfu and Bauschke, Heinz H.},
	title		= {The {B}regman proximal average},
	year		= {2022},
	journal		= {SIAM Journal on Optimization},
	volume		= {32},
	pages		= {1379--1401},
	number		= {2},
	publisher	= {SIAM},
}

@article{wang2024mirror,
	author		= {Wang, Ziyuan and Themelis, Andreas and Ou, Hongjia and Wang, Xianfu},
	title		= {A Mirror Inertial Forward–Reflected–Backward Splitting: Convergence Analysis Beyond Convexity and {L}ipschitz Smoothness},
	year		= {2024},
	journal		= {Journal of Optimization Theory and Applications},
	volume		= {203},
	pages		= {1127--1159},
	number		= {2},
	issn		= {1573-2878},
	doi			= {10.1007/s10957-024-02383-9},
}

@article{wang2025bregman,
	author		= {Wang, Ziyuan and Themelis, Andreas},
	title		= {{B}regman level proximal subdifferentials and new characterizations of {B}regman proximal operators},
	year		= {2026},
	journal		= {Set-Valued and Variational Analysis \upshape{(to appear)}},
}

@article{yagishita2025simple,
	author		= {Yagishita, Shotaro and Ito, Masaru},
	title		= {Simple linesearch-free first-order methods for nonconvex optimization},
	year		= {2025},
	journal		= {\arXivLink{2509.14670}},
}

@article{zhou2025adabb,
	author		= {Zhou, Danqing and Ma, Shiqian and Yang, Junfeng},
	title		= {{AdaBB}: Adaptive {B}arzilai-{B}orwein Method for Convex Optimization},
	year		= {2025},
	journal		= {Mathematics of Operations Research},
	doi			= {10.1287/moor.2024.0510},
}
